\documentclass{amsart}
\usepackage{amssymb}
\usepackage{amsmath}
\usepackage{latexsym}
\usepackage{graphicx}
\usepackage{amscd}
\usepackage{eufrak}
\usepackage{mathrsfs}
\usepackage[english]{babel}
\usepackage{verbatim}

\newtheorem{theorem}[equation]{Theorem}
\newtheorem{thm}[equation]{Theorem}
\newtheorem{theorem-definition}[equation]{Theorem-Definition}
\newtheorem{lemma-definition}[equation]{Lemma-Definition}
\newtheorem{definition-prop}[equation]{Proposition-Definition}

\newtheorem{prop}[equation]{Proposition}
\newtheorem{lemma}[equation]{Lemma}
\newtheorem{cor}[equation]{Corollary}
\newtheorem{definition}[equation]{Definition}

\theoremstyle{definition}
\newtheorem{exam}[equation]{Example}

\newcommand{\llbr}{[\negthinspace[}
\newcommand{\rrbr}{]\negthinspace]}

\newcommand{\llpar}{(\negthinspace(}
\newcommand{\rrpar}{)\negthinspace)}

\theoremstyle{definition}
\newtheorem{example}[equation]{Example}

\newtheorem{rema}[equation]{Remark}
\newtheorem{remark}[equation]{Remark}

\newcommand{\LL}{\ensuremath{\mathbb{L}}}

\newcommand{\Z}{\ensuremath{\mathbb{Z}}}
\newcommand{\Q}{\ensuremath{\mathbb{Q}}}

\newcommand{\R}{\ensuremath{\mathbb{R}}}
\newcommand{\C}{\ensuremath{\mathbb{C}}}

\newcommand{\A}{\ensuremath{\mathbb{A}}}

\newcommand{\mU}{\ensuremath{\mathfrak{U}}}
\newcommand{\mV}{\ensuremath{\mathfrak{V}}}
\newcommand{\mX}{\ensuremath{\mathfrak{X}}}

\newcommand{\cA}{\ensuremath{\mathscr{A}}}
\newcommand{\cB}{\ensuremath{\mathscr{B}}}

\newcommand{\cE}{\ensuremath{\mathscr{E}}}

\newcommand{\cP}{\ensuremath{\mathscr{P}}}
\newcommand{\cT}{\ensuremath{\mathscr{T}}}
\newcommand{\cU}{\ensuremath{\mathscr{U}}}
\newcommand{\cV}{\ensuremath{\mathscr{V}}}
\newcommand{\cW}{\ensuremath{\mathscr{W}}}
\newcommand{\cX}{\ensuremath{\mathscr{X}}}
\newcommand{\cY}{\ensuremath{\mathscr{Y}}}
\newcommand{\cZ}{\ensuremath{\mathscr{Z}}}

\renewcommand{\R}{\ensuremath{\mathbb{R}}}
\renewcommand{\C}{\ensuremath{\mathbb{C}}}

\renewcommand{\A}{\ensuremath{\mathbb{A}}}

\renewcommand{\cA}{\ensuremath{\mathscr{A}}}

\renewcommand{\cU}{\ensuremath{\mathscr{U}}}

\renewcommand{\cZ}{\ensuremath{\mathscr{Z}}}
\renewcommand{\cY}{\ensuremath{\mathscr{Y}}}

\newcommand{\Spec}{\ensuremath{\mathrm{Spec}\,}}

\newcommand{\Lie}{\mathrm{Lie}}

\newcommand{\ord}{\mathrm{ord}}
\newcommand{\red}{\mathrm{red}}

\newcommand{\Var}{\mathrm{Var}}
\newcommand{\Gal}{\mathrm{Gal}}
\newcommand{\Hom}{\mathrm{Hom}}
\newcommand{\coker}{\mathrm{coker}}

\newcommand{\wt}{\mathrm{wt}}

\newcommand{\lct}{\mathrm{lct}}
\newcommand{\Sk}{\mathrm{Sk}}

\newcommand{\an}{\mathrm{an}}
\newcommand{\rig}{\mathrm{rig}}

\newcommand{\spe}{\mathrm{sp}}

\newcommand{\Sm}{\mathrm{Sm}}

\newcommand{\gro}{\mathcal{M}_k^{\widehat{\mu}}}

\newcommand{\loga}[1]{\ensuremath{\mathscr{#1}^{\dagger}}}

\numberwithin{equation}{subsection}

\newcommand{\sss}{\vspace{5pt} \subsubsection*{ }\refstepcounter{equation}{{\bfseries(\theequation)}\ }}

\begin{document}
\title{Motivic zeta functions of degenerating Calabi-Yau varieties}

\author[Lars Halvard Halle]{Lars Halvard Halle}
\address{University of Copenhagen\\
Department of Mathematical Sciences\\
Universitetsparken 5\\
2100 Copenhagen\\
Denmark}
\email{larshhal@math.ku.dk}

\author[Johannes Nicaise]{Johannes Nicaise}
\address{Imperial College\\
Department of Mathematics\\ South Kensington Campus \\
London SW7 2AZ, UK} \email{j.nicaise@imperial.ac.uk}

\begin{abstract}
We study motivic zeta functions of degenerating families of Calabi-Yau varieties. Our main result says that they satisfy an analog of Igusa's monodromy conjecture
if the family has a so-called Galois equivariant Kulikov model; we provide several classes of examples where this condition is verified. We also establish a close relation between the zeta function and the skeleton that appeared in Kontsevich and Soibelman's non-archimedean interpretation of the SYZ conjecture in mirror symmetry.
\end{abstract}

\maketitle

\section{Introduction}

Let $X$ be a geometrically connected, smooth and proper variety over $ K = \mathbb{C}\llpar t\rrpar $. We say that $X$ is \emph{Calabi-Yau} if the canonical line bundle $\omega_X$ is trivial. For every volume form $\omega$ on $X$, one can associate to the pair $ (X,\omega) $ an invariant $Z_{X,\omega}(T)$ called the motivic zeta function. This is a formal power series in $T$ with coefficients in a certain Grothendieck ring of varieties, which encodes a wealth of information about the degeneration of $X$ at $t=0$.
 It arose as a natural analog of Denef and Loeser's motivic zeta function for hypersurface singularities; the precise relation is explained in \cite{HaNi-CY}.

 The main open question regarding Denef and Loeser's motivic zeta functions for hypersurface singularities is the so-called \emph{monodromy conjecture}, which predicts that each pole of the zeta function should correspond to a local monodromy eigenvalue in a precise way. This intriguing conjecture has motivated a substantial amount of interesting work, and has been proved in some significant cases, but is still wide open in general.
  It is natural to wonder whether there exists a similar relation between motivic zeta functions and monodromy eigenvalues also in the context of degenerating Calabi-Yau varieties. We say that $X$ satisfies the \emph{Monodromy Property} if each pole of $Z_{X,\omega}(T)$ gives rise to a monodromy eigenvalue (see Definition \ref{def:MP} for a precise statement). We proved in \cite{HaNi} that this is indeed the case when $X$ is an abelian variety; in fact, we showed that  the zeta function of an abelian variety has only one pole, which can be explicitly related to a certain arithmetic invariant of the abelian variety (Chai's base change conductor).

   In the present article, we investigate to what extent these results generalize to arbitrary Calabi-Yau varieties. The problem becomes much more intricate because the techniques for abelian varieties (N\'eron models) are no longer available, and the zeta functions tend to have more poles than in the abelian case. In fact, very few examples have been studied beyond the abelian case, and this paper is the first systematic effort to understand the properties of zeta functions of Calabi-Yau varieties that are not abelian (a few results have been announced without proofs in our survey paper \cite{HaNi-CY}).
 We also refine our results in the abelian case to take the monodromy action on the zeta function into account, which was ignored in \cite{HaNi}. This requires some surprisingly subtle new results on Galois actions on N\'eron models.

 It turns out that the main properties of the unique pole of the zeta function of an abelian variety remain valid for the {\em largest} pole of the zeta function of a general Calabi-Yau variety, although new techniques need to be used for the proofs. In particular, we use Hodge theory to prove in Theorem \ref{thm:hodge} that the largest pole always gives rise to a monodromy eigenvalue, as predicted by the monodromy conjecture.
 We also explain the close relation that exists between this largest pole and the properties of the {\em essential skeleton} of $X$, an object that was introduced by Kontsevich and Soibelman in their non-archimedean interpretation of the SYZ fibration in mirror symmetry \cite{KS}. Specifically, we prove in Theorem \ref{thm:largest} that the order of this pole is equal to one plus the dimension of the essential skeleton, and the pole itself is equal to one minus the minimal value of the {\em weight function} associated with $\omega$, the function whose locus of minimal values is, by definition, the essential skeleton.

 Some of the key properties of N\'eron models that were used in \cite{HaNi} were the triviality of the relative canonical line bundle, their Galois-equivariance
 (due to their canonical nature) and Grothendieck's Semi-Stable Reduction Theorem.
 As a (partial) replacement for N\'eron models, we consider in Section \ref{sec:kulikov} another type of Galois equivariant minimal models. More precisely, we assume that our Calabi-Yau variety $X$ admits, after a finite extension of the base field $K$, a proper and regular model $\mathscr{Y} $ over the valuation ring such that $\mathscr{Y}_k$ is a normal crossings divisor, the logarithmic relative canonical bundle is trivial and the natural Galois action in the generic fiber extends to $\mathscr{Y} $. Inspired by Kulikov's seminal work on semi-stable models of degenerating K3 surfaces, we say that $\mathscr{Y}$ is an \emph{equivariant Kulikov model} of $X$.

Theorem \ref{thm:main} asserts that if $X$ admits an equivariant Kulikov model $\mathscr{Y} $, then $Z_{X,\omega}(T)$ has a unique pole. Combined with Theorem \ref{thm:hodge}, this implies that $X$ has the Monodromy Property. Let us briefly indicate the main ideas of our approach. First we use techniques due to Gabber (see \cite{gabber}) in order to modify $\mathscr{Y} $ into a model on which the Galois action is \emph{very tame}, by endowing $\cY$ with a suitable log structure. The significance of having a very tame action is that the quotient (in the category of log schemes) forms a log smooth $R$-model of $X$. The motivic zeta function can then be computed by means of the explicit formula obtained by Bultot and the second author \cite{BuNi}. In practice, some refinement of the above mentioned approach is needed: to obtain maximal flexibility, it is natural to allow Kulikov models to be algebraic spaces (as in Kulikov's work). This makes it necessary to reformulate some of Gabber's results in the framework of algebraic spaces. We moreover upgrade a few fundamental constructions in motivic integration to algebraic spaces. These results are presented separately, as an appendix, in Section \ref{sec:appendix}, and should be of interest also outside the applications in this paper. In particular, we answer a question by Stewart and Vologodsky.

It should be pointed out that it is not always true that a Calabi-Yau variety $X$ admits an equivariant Kulikov model, even in the case of K3 surfaces. Indeed, it is an immediate consequence of Theorem \ref{thm:main} that as soon as the motivic zeta function has more than one pole, such models for $X$ can not exist after any base extension. However, contrary to what one finds in the abelian case, motivic zeta functions of Calabi-Yau varieties can very well have more than one pole. For instance, in  \S\ref{ex:nokulikov}, we give a basic example of a quartic K3 surface with potential good reduction where the motivic zeta function has two poles (but still satisfies the Monodromy Property).
 In other words, the algebraic complexity of the series $Z_{X,\omega}(T)$ imposes an obstruction for the existence of minimal models of $X$ carrying `extra' symmetry. It is an interesting problem in itself to understand when a given Calabi-Yau variety $X$ admits an equivariant Kulikov model. The reader will find a discussion on related matters in the context of good reduction of K3 surfaces in \cite{liedtke-matsumoto}, but, to our best knowledge, the literature on this question is sparse. In Section \ref{sec:examples}, we discuss a few examples of Calabi-Yau varieties with potential good reduction where an equivariant Kulikov model can indeed be found.

 To conclude the introduction, we give a brief overview of the structure of the paper.
 In Section \ref{sec:motzeta} we fix our notations and present some preliminary material on motivic zeta functions. In Section \ref{sec:largest}
 we study the largest pole of the zeta function, in particular, the relation with the essential skeleton and with Hodge theory, and we prove that it always satisfies the Monodromy Property. Section \ref{sec:abelian} contains the refinements of our results for abelian varieties that are necessary to take the monodromy action into account; we also prove that the essential skeleton coincides with Berkovich's skeleton for abelian varieties using the interpretation of the weight function by M.~Temkin.
 The heart of the paper is Section \ref{sec:kulikov}, where we define equivariant Kulikov models and prove that the zeta function has a unique pole whenever such a model exists, and thus satisfies the Monodromy Property.
  In Section \ref{sec:examples}, we construct various interesting examples where equivariant Kulikov models indeed exist. Finally, Section \ref{sec:appendix} is an appendix containing the technical results on motivic integration on algebraic spaces that are needed elsewhere in the paper.

\subsection*{Acknowledgements} We are grateful to David Rydh for useful discussions and for suggesting the approach in Section \ref{subsec-equidef}.
 We are also endebted to Chenyang Xu and Klaus K{\"u}nnemann for answering our questions on minimal models for $K3$ surfaces and Mumford models of abelian varieties, respectively. JN is supported by the ERC Starting Grant MOTZETA (project 306610) of the European Research Council.

\section{Motivic zeta functions}\label{sec:motzeta}

\subsection{Preliminaries}
 \sss We set $k=\C$,  $R=k\llbr t\rrbr$ and $K=k\llpar t\rrpar$. We denote by $\mathrm{ord}_t:K^{\times}\to \Z$ the $t$-adic valuation on $K$, and we define an absolute value
 $|\cdot|_K$ on $K$ by setting $|a|_K=\exp(-\mathrm{ord}_ta)$ for every $a\in K^{\times}$. This turns $K$ into a non-archimedean complete valued field.
  We denote by $(\cdot)^{\an}$ the analytification functor from the category of $K$-schemes of finite type to Berkovich's category of $K$-analytic spaces.
 We set $S=\Spec R$ and we will denote by $S^{\dagger}$ the log scheme that we obtain by endowing $S$ with its standard log structure (the divisorial log structure induced by the closed point). Unless explicitly stated otherwise, all the log structures in this paper are \'etale.

\sss For every positive integer $d$, we set $R(d)=k\llbr \sqrt[d]{t}\rrbr$ and $K(d)=k\llpar \sqrt[d]{t}\rrpar$. The union of the fields $K(d)$ is an algebraic closure of $K$ that we denote by $K^a$.
 We denote by $\mu_d$ the group of $d$-th roots of unity in $k$ and by
$$\widehat{\mu}=\lim_{\longleftarrow}\mu_d$$ the profinite group of roots of unity in $k$.
 The field $K(d)$ is a Galois extension of $K$ and its Galois group is canonically isomorphic to $\mu_d$.
 We let $\mu_d$ act on $K(d)$ from the right via the inverse of the Galois action over $K$:
 $$K(d)\times \mu_d\to K(d):(\psi(\sqrt[d]{t}),\zeta)\mapsto \psi(\zeta^{-1}\sqrt[d]{t}).$$
  We denote by $\sigma$ the canonical topological generator of the inertia group $\Gal(K^a/K)\cong \widehat{\mu}$, that is, $\sigma=(\exp(2\pi i/d))_{d>0}$. We call the element $\sigma$ the {\em monodromy operator}.

\sss Let $X$ be a proper $K$-scheme.
 A model for $X$ over $R$ is a flat $R$-scheme $\cX$ endowed with an isomorphism of $K$-schemes $\cX_K\to X$.
 If $X$ is smooth over $K$, we say that $\cX$ is an snc-model for $X$ if it is regular and proper over $R$, and the special fiber $\cX_k$ is a strict normal crossings divisor on $\cX$. Such a model always exists, by Hironaka's resolution of singularities. We say that $X$ has semi-stable reduction if it has an snc-model with reduced special fiber; such a model is called a semi-stable model of $X$. For every proper model $\cX$ over $R$, there exist a positive integer $d$ and a semi-stable model $\cX'$ of $X\times_K K(d)$ that dominates $\cX\times_R R(d)$, by the semi-stable reduction theorem \cite[Ch4\S3]{KKMS}. If $X$ is projective, then we can take $\cX'$ to be projective over $R(d)$, as well.

\sss Let $X$ be a proper $K$-scheme. Then, for every $i\geq 0$, there exists a canonical $\Q$-vector space $H^i(X\times_K K^a,\Q)$ endowed with a quasi-unipotent action of $\sigma$
such that for every prime $\ell$, the tensor product with $\Q_\ell$ is canonically isomorphic to the $\ell$-adic cohomology space $H^i(X\times_K K^a,\Q_\ell)$ endowed with the Galois action of $\sigma$. This follows from Berkovich's theory of \'etale cohomology with $\Z$-coefficients for analytic spaces over $K$ \cite[7.1.1]{berk-Z}, which relies on Kato and Nakayama's theory of Betti cohomology for logarithmic complex analytic spaces \cite{KaNa}.
 Berkovich's construction provides the following comparison result with complex analytic nearby cohomology. Let $R_0$ be the ring of germs of holomorphic functions on $\C$ at the origin, and denote its fraction field by $K_0$.
  We choose a uniformizer $t$ in $R_0$. This choice determines an isomorphism between $R$ and the completion of $R_0$.
  If $X$ is defined over $K_0$ and $\cX$ is a proper $R_0$-model for $X$, then $H^i(X\times_{K_0} K^a,\Q_\ell)$ is canonically isomorphic to
 $$\mathbb{H}^i(\cX_k(\C),R\psi_{\cX}(\Q))$$
 where $R\psi_{\cX}(\Q)$ is the complex analytic nearby cycles complex associated with $\cX$.
 Under this isomorphism, the action of $\sigma$ corresponds to the monodromy action on $\mathbb{H}^i(\cX_k(\C),R\psi_{\cX}(\Q))$.

 \sss All algebraic spaces are assumed to be separated and Noetherian. In Section \ref{sec:kulikov} we will need to work with log algebraic spaces; we will freely use the definitions and results for log schemes if they carry over to log algebraic spaces by standard descent arguments. For instance, we can speak about regular  log algebraic spaces and smooth morphisms of log algebraic spaces because these notions are local with respect to the \'etale topology.
   If additional care is required in the context of algebraic spaces, this will be clearly explained in the text.

\subsection{Galois-equivariant motivic integrals}\label{ss:motint}
\sss Unless explicitly stated otherwise, we assume that groups act on schemes from the left. For every positive integer $d$, we say that an action of $\mu_d$ on a $k$-scheme of finite type $X$ is {\em good} if
 $X$ has a finite partition into $\mu_d$-stable affine subschemes.
    We say that an action of $\widehat{\mu}$ on $X$ is good if it factors through a good action of $\mu_d$ on $X$, for some $d>0$.
 The Grothendieck group $K^{\widehat{\mu}}_0(\Var_k)$ of $k$-varieties with $\widehat{\mu}$-action is the abelian group defined by the following presentation:
 \begin{itemize}
\item {\em Generators:} Isomorphism classes $[X]$ of $k$-schemes of finite type $X$ endowed with a good action of $\widehat{\mu}$; here the isomorphism class is taken with respect to $\widehat{\mu}$-equivariant isomorphisms.
\item{\em Relations:}
\begin{enumerate}
\item If $X$ is a $k$-scheme of finite type with good $\widehat{\mu}$-action and $Y$ is a closed subscheme of $X$ that is stable under the $\widehat{\mu}$-action, then
$$[X]=[Y]+[X\setminus Y].$$
\item If $X$ is a $k$-scheme of finite type with good $\widehat{\mu}$-action and $A\to X$ is an affine bundle of rank $r$ endowed with an affine lift of the $\widehat{\mu}$-action on $X$, then
$$[A]=[\A^r_{k}\times_{k} X]$$ where $\widehat{\mu}$ acts trivially on $\A^r_{k}$.
\end{enumerate}
\end{itemize}
 Here, ``affine lift'' means that we have an action of $\widehat{\mu}$ on $A$ and a $\widehat{\mu}$-linear action on the underlying vector bundle $V$ of translations of $A$ such that the morphism $A\to X$ and the action of $V$ on $A$ are $\widehat{\mu}$-equivariant.
  We define a ring structure on $K^{\widehat{\mu}}_0(\Var_k)$ by means of the multiplication rule $[X]\cdot [X']=[X\times_k X']$ where $\widehat{\mu}$ acts diagonally on $X\times_k X'$. We write $\LL$ for the class of the affine line $\A^1_{k}$ with trivial $\widehat{\mu}$-action and we set $\mathcal{M}^{\widehat{\mu}}_k=K^{\widehat{\mu}}_0(\Var_k)[\LL^{-1}]$.

\begin{remark} Our definition of a good $\mu_d$-action on $X$ is weaker than the one that is commonly used, namely, that $X$ can be covered by $\mu_d$-stable affine open subschemes. Our definition has the advantage that it can be generalized to algebraic spaces, and it gives rise to the same equivariant Grothendieck ring as the usual definition.
\end{remark}

 \sss Let $X$ be a smooth and proper $K$-scheme. For every $d>0$, we set $X(d)=X\times_K K(d)$. The group $\mu_d$ acts on $X(d)$ from the left.
 An {\em equivariant weak N\'eron model} for $X(d)$ is a separated smooth $R(d)$-scheme $\cY$, endowed with a good $\mu_d$-action and a $\mu_d$-equivariant isomorphism of
 $K(d)$-schemes $\cY_{K(d)}\to X(d)$, such that the natural map $\cY(R(d))\to X(K(d))$ is a bijection. Such an equivariant weak N\'eron model always exists: starting from
 a proper $R$-model $\cX$ of $X$, we can apply the smoothening algorithm in the proof of \cite[3.4.2]{BLR} to $\cX\times_R R(d)$. Inspecting the different steps of the algorithm, one sees that it produces a morphism of $R(d)$-schemes $\cY'\to \cX\times_R R(d)$ that is a finite composition of $\mu_d$-equivariant blow-ups with centers in the special fiber such that the $R(d)$-smooth locus $\cY$ of $\cY'$ is an equivariant weak N\'eron model of $X(d)$. Note that the $\mu_d$-action on $\cY$ is good because the morphism
 $\cY\to \cX\times_R R(d)$ is quasi-projective.

 \sss Now assume that the canonical line bundle of $X$ is trivial, and let $\omega$ be a volume form on $X$. We write $\omega(d)$ for the pullback of $\omega$ to $X(d)$.
 For every connected component $C$ of $\cY_k$, we denote by $\ord_{C}\omega(d)$ the unique integer $a$ such that $t^{-a/d}\omega$ is a generator
 for $\omega_{\cY/R(d)}$ locally at the generic point of $C$. For every integer $i$, we denote by $C(i)$ the union of the connected components $C$ of $\cY_k$
  such that $\ord_{C}\omega(d)=i$. This union is stable under the action of $\mu_d$ on $\cY_k$, because $\omega$ is defined over $K$.

 \begin{prop}
 The element
  \begin{equation}\label{eq:motint}
  \sum_{i\in \Z}[C(i)]\LL^{-i}
  \end{equation}
  of $\mathcal{M}_k^{\widehat{\mu}}$ does not depend on the choice of the equivariant weak N\'eron model $\cY$ of $X(d)$.
 \end{prop}
 \begin{proof}
 Any two equivariant weak N\'eron models of $X(d)$ can be dominated by a third (apply the smoothening algorithm of \cite[3.4.2]{BLR}
 to the schematic image of the diagonal embedding of $X(d)$ in the product of the two models). Thus the result follows from the change of variables formula for equivariant
 motivic integrals: see \cite[6.4]{hartmann}.
 \end{proof}

 \begin{definition}
 We denote the element \eqref{eq:motint} by
 $$\int_{X(d)}|\omega(d)| \quad \in \mathcal{M}_k^{\widehat{\mu}}$$
 and call it the motivic integral of $\omega(d)$ on $X(d)$.
 \end{definition}

\subsection{Motivic zeta functions}

\begin{definition}
 Let $X$ be a geometrically connected, smooth and proper $K$-scheme with trivial canonical line bundle, and let $\omega$ be a volume form on $X$.
 We define the motivic zeta function of the pair $(X,\omega)$ as
 $$Z_{X,\omega}(T)=\sum_{d>0}\left(\int_{X(d)}|\omega(d)|\right)T^d\quad \in \mathcal{M}_k^{\widehat{\mu}}\llbr T\rrbr.$$
\end{definition}

\sss \label{sss:snc} This motivic zeta function can be explicitly computed in the following way. Let $\cX$ be an snc-model of $X$. We write
 $\cX_k=\sum_{i\in I}N_i E_i.$
  The volume form $\omega$ defines a rational section of the logarithmic relative canonical line bundle $\omega_{\cX/R}(\cX_{k,\red}-\cX_k)$ of $\cX$.
   We denote the associated Cartier divisor by $\mathrm{div}_{\cX}(\omega)$. It is supported on $\cX_k$; we write it as
   $\mathrm{div}_{\cX}(\omega)=\sum_{i\in I}\nu_i E_i.$
    For every non-empty subset $J$ of $I$, we set $$E_J=\bigcap_{j\in J}E_j,\quad E_J^o=E_J\setminus \left(\bigcup_{i\notin J}E_i\right).$$
 We set $N_J=\gcd\{N_j\,|\,j\in J\}$ and we denote by $\cX(N_J)$ the normalization of $\cX\times_R R(N_J)$. Then the
 scheme $\widetilde{E}^o_J=\cX(N_J)\times_{\cX} E_J^o$ is a Galois cover of $E_J^o$ that is described explicitly in \cite[2.3]{Ni-tameram}.
  The group $\mu_{N_J}$ acts on $\widetilde{E}^o_J$ {\em via} its action on $R(N_J)$.

\begin{theorem}\label{thm-snc}
With the above notations, we have
$$Z_{X,\omega}(T)=\sum_{\emptyset \neq J\subset I}[\widetilde{E}^o_J](\LL-1)^{|J|-1}\prod_{j\in J}\frac{\LL^{-\nu_j}T^{N_j}}{1-\LL^{-\nu_j}T^{N_j}}\quad \in \mathcal{M}_k^{\widehat{\mu}}\llbr T\rrbr.$$
\end{theorem}
\begin{proof}
If we forget the $\widehat{\mu}$-action, this follows from \cite[7.7]{NiSe} (where a different normalization of the motivic measure was used, resulting in the extra factor $\LL^{-\mathrm{dim}(X)}$). The formula with $\widehat{\mu}$-action follows from \cite[6.2.1]{BuNi}.
\end{proof}

\sss We have considered a motivic generating series similar to $Z_{X,\omega}(T)$ (for a specific choice of $\omega$ and forgetting the $\widehat{\mu}$-action) in \cite[\S2.4]{HaNi} and \cite[\S6]{HaNi-CY}.
 It is explained in \cite{HaNi} how $Z_{X,\omega}(T)$ can be viewed as an analog of Denef and Loeser's motivic zeta function of a hypersurface singularity \cite{DL-barc}, based on the results in \cite{NiSe}.
 The most important problem about Denef and Loeser's motivic zeta function is the so-called {\em monodromy conjecture}, which predicts a precise relation between poles of the zeta function and local monodromy eigenvalues of the hypersurface \cite[4.17]{HaNi-CY}. It is natural to ask if the analogous property holds for $Z_{X,\omega}(T)$. The following formulation is a refinement of \cite[2.6]{HaNi} and \cite[6.9]{HaNi-CY} (adding the $\widehat{\mu}$-action).

 \begin{definition}\label{def:MP}
    We say that $X$ satisfies the Monodromy Property if there exists a finite set $S$ of rational numbers such that
 $Z_{X,\omega}(T)$ belongs to the subring
 $$\gro \left[T,\frac{1}{1-\LL^a T^b}\right]_{(a,b)\in \Z\times \Z_{>0},\,a/b\in S}$$
 of $\gro\llbr T\rrbr$ and such that, for every $s\in S$, the number $\exp(2\pi i s)$ is an eigenvalue of $\sigma$ on
 $$H^i(X\times_K K^a,\Q)$$ for some $i\geq 0$.
  \end{definition}

\sss This property does not depend on the volume form $\omega$: it follows immediately from the definition of the motivic zeta function that
\begin{equation}\label{eq:rescale}
Z_{X,c\omega}(T)=Z_{X,\omega}(\LL^{-\mathrm{ord}_tc}T)
\end{equation} for every element $c$ of $K^{\times}$.
 The main result of \cite{HaNi} states that $X$ satisfies the Monodromy Property if $X$ is an abelian variety and we forget the $\widehat{\mu}$-action on $Z_{X,\omega}(T)$, but otherwise, little is known.
  We will prove in Section \ref{sec:kulikov} that $X$ satisfies the Monodromy Property if it has a so-called {\em equivariant Kulikov model}.
  This condition is not always satisfied (see Example \ref{ex:nokulikov}) but we do not know any example where $X$ does not satisfy the Monodromy Property.

\sss \label{sss:acampo} In order to check the Monodromy Property in concrete examples, it is often useful to use a variant of A'Campo's formula for the monodromy zeta function that was proven in \cite[2.6.2]{Ni-tameram}. Let $X$ be a smooth and proper $K$-scheme. Let $\cX$ be an snc-model for $X$, with $\cX_k=\sum_{i\in I}N_i E_i$, and define $E_i^o$ as in \eqref{sss:snc}. Then
$$\prod_{n\geq 0} \det(t\cdot \mathrm{Id}-\sigma\,|\,H^n(X\times_K K^a,\Q))^{(-1)^{n+1}}=\prod_{i\in I}(t^{N_i}-1)^{-\chi(E_i^o)}$$
where $\chi$ is the topological Euler characteristic.  The left hand side of this expression is called the {\em monodromy zeta function} of $X$.

 \section{The largest pole}\label{sec:largest}
\subsection{The weight function and the essential skeleton}
\sss \label{sss:ess} Let $X$ be a geometrically connected, smooth and proper $K$-scheme with trivial canonical line bundle, and let $\omega$ be a volume form on $X$.
 The {\em essential skeleton} $\Sk(X)$ of $X$ was constructed in \cite{KS} to serve as the base of their non-archimedean SYZ fibration in the theory of mirror symmetry.
  The construction was further developed in \cite{MuNi}, \cite{NiXu} and \cite{NiXu2}. We will briefly summarize some of its main properties.
  To the volume form $\omega$, one can associate a so-called {\em weight function} \cite[\S4]{MuNi}
$$\wt_{\omega}:X^{\an}\to \R\cup \{+\infty\}.$$
 The essential skeleton $\Sk(X)$ is the subspace of $X^{\an}$ consisting of the points where $\wt_{\omega}$ takes its minimal value; we will call this minimal value the minimal  weight of $\omega$ on $X$ and denote it by $\min(\omega)$. The essential skeleton does not depend on the choice of $\omega$ because $\wt_{a\omega}=\wt_{\omega}+\mathrm{ord}_t(a)$ for every $a\in K^\times$.
  It can be computed in the following way. Let $\cX$ be an snc-model of $X$ over $R$.
  We write $\cX_k=\sum_{i\in I}N_i E_i$ and $\mathrm{div}_{\cX}(\omega)=\sum_{i\in I}\nu_i E_i$ as in \eqref{sss:snc}.
 Then the dual intersection complex of $\cX_k$ can be embedded into $X^{\an}$ in a canonical way; it is called the {\em Berkovich skeleton} of the model $\cX$ and denoted by $\Sk(\cX)$ (see for instance \cite[\S3]{MuNi}). The weight function $\wt_\omega$ is affine on every face of $\Sk(\cX)$, and its value at the vertex $v_i$ corresponding to an irreducible component $E_i$ of $\cX_k$ is given by $(\nu_i/N_i)+1$ (the shift by $1$ is caused by the fact that we worked with $\omega_{\cX/R}(\cX_{k,\red})$, rather than
 $\omega_{\cX/R}(\cX_{k,\red}-\cX_k)$, in \cite{MuNi}). The essential skeleton $\Sk(X)$ of $X$ is the subcomplex of $\Sk(\cX)$ spanned by the vertices $v_i$ for which
 $\nu_i/N_i$ is minimal \cite[4.7.5]{MuNi}; the minimal weight of $\omega$ on $X$ is given by the formula $$\min(\omega)=\min\{ \frac{\nu_i}{N_i}+1\,|\,i\in I\}.$$
It is proven in \cite[4.2.4]{NiXu} that $\Sk(X)$ is a strong deformation retract of $X^{\an}$.

\sss The essential skeleton of $X$ can also be computed from a minimal model of $X$ in the sense of the Minimal Model Program. It is proven in Corollary 4 of \cite{KNX} that
 $X$ has a minimal qdlt-model $\cX_{\min}$ over $R$.
  One can still associate a Berkovich skeleton $\Sk(\cX_{\min})$ to such a qdlt-model, by \cite[\S23]{KNX}. It is shown in
Theorem 24 of \cite{KNX} that $\Sk(\cX_{\min})$ coincides with the essential skeleton $\Sk(X)$. If $X$ is defined over an algebraic $k$-curve, rather than the field of Laurent series, one can even find a minimal dlt-model; in practice, it is often possible to reduce to this case by means of the approximation technique in \cite[\S4.2]{NiXu}.
 The essential skeleton behaves well under finite base change, as explained by the following proposition.

 \begin{prop}\label{prop:bc}
Let $d$ be a positive integer and set $X(d)=X\times_K K(d)$. Denote by  $\omega(d)$ the pullback of $\omega$ to $X(d)$ and by $\pi$ the canonical base change morphism
 $\pi:X(d)^{\an}\to X^{\an}$. Then $$\wt_{\omega(d)}=(\wt_{\omega}\circ \pi)+1-d.$$ In particular,
  the essential skeleton $\Sk(X(d))$ is the inverse image of $\Sk(X)$ under the
 canonical morphism $X(d)^{\an}\to X^{\an}$. Thus $\Sk(X(d))$ is stable under the Galois action of $\mu_d(k)$, and $\Sk(X)$ is canonically homeomorphic to the quotient
 $\Sk(X(d))/\mu_d(k)$.
  If $X$ has semi-stable reduction, then the map $\Sk(X(d))\to \Sk(X)$ is a homeomorphism.
 \end{prop}
 \begin{proof}
 The expression for the weight function $\wt_{\omega(d)}$ is an easy consequence of the compatibility of relative log differentials with base change; see \cite[4.1.9]{NiXu} and its proof.
  Thus it is enough to prove that, if $X$ has a semi-stable model $\cX$, then
  the corestriction of
  $X(d)^{\an}\to X^{\an}$ over $\Sk(\cX)$ is bijective. This implies final assertion in the statement, because $\Sk(X)$ is contained in $\Sk(\cX)$.
  It suffices to show that $K$ is algebraically closed in the completed residue field $\mathscr{H}(x)$ of $X^{\an}$ at any point $x$ of $\Sk(\cX)$, because the
  fiber of $X(d)^{\an}\to X$ over $x$ is canonically isomorphic to the spectrum of the Banach algebra $\mathscr{H}(x)\otimes_K K(d)$.
    Assume the contrary; then there exist an element $s$ in $\mathscr{H}(x)$ and an integer $m\geq 2$ such that $s^m=t$. The element $s$ already belongs to the residue field of the local ring $\mathcal{O}_{X^{\an},x}$, because this residue field is quasi-complete \cite[2.3.3 and 2.4.1]{berk-etale}. Since $\mathcal{O}_{X^{\an},x}$ is henselian \cite[2.1.5]{berk-etale}, it now follows that $t$ has an $m$-th root on an open neighbourhood of $x$ in $X^{\an}$. Thus we may assume that $x$ is divisorial in the sense of \cite[\S2.4]{MuNi}, because divisorial points are dense in $X^{\an}$ \cite[2.4.9]{MuNi}.

  Let $y$ be the image of $x$ under the specialization map $\spe_{\cX}:X^{\an}\to \cX_k$ and let $(z_1,\ldots,z_r)$ be a regular system of local parameters in $\mathcal{O}_{\cX,y}$ such that $t=z_1\cdots z_r$. Since $x$ lies in the skeleton $\Sk(\cX)$, it is monomial with respect to the model $\cX$, so that there exists a tuple of integers $a=(a_1,\ldots,a_r)$ such that $z^a=z_1^{a_1}\cdots z_r^{a_r}$ is a uniformizer in the valuation ring $\mathscr{H}(x)^o$ of $\mathscr{H}(x)$.
  Denote by $\kappa(y)$ the residue field of $\cX$ at $y$. The choice of a section of the projection $\widehat{\mathcal{O}}_{\cX,y}\to \kappa(y)$ determines an isomorphism $\widehat{\mathcal{O}}_{\cX,y}\cong \kappa(y)\llbr z_1,\ldots,z_r \rrbr$, and $\mathscr{H}(x)$ is the completion of the fraction field of $\widehat{\mathcal{O}}_{\cX,y}$ with respect to the divisorial valuation $v_x$ corresponding to the point $x$. The residue field $\widetilde{\mathscr{H}}(x)$ of the valued field $\mathscr{H}(x)$
  is given by $\kappa(y)(z^{b_1},\ldots,z^{b_{r-1}})$ where $\{a,b_1,\ldots,b_{r-1}\}$ is a basis of $\Z^r$. If we denote by $\gamma$ the normalized valuation of $s$ in the discretely valued field $\mathscr{H}(x)$, then $s/z^{\gamma a}$ is a unit in $\mathscr{H}(x)^o$ and the $m$-th power of its reduction in $\widetilde{\mathscr{H}}(x)$ coincides with the reduction of $t/z^{m\gamma a}$.
  This is impossible, because the element $(1,\ldots,1)-m\gamma a$ of $\Z^r$ is not divisible by $m$.
  \if false
  This valuation is determined by the weights $w_i=-\ln|z_i(x)|\in (0,1]$ for $i\in \{1,\ldots,r\}$. We denote by $w_{\min}$ the minimum of the weights $w_i$.

  Now it is easy to see that $t$ has no roots of order $m\geq 2$ in $\mathscr{H}(x)$, so that $K$ is algebraically closed in $\mathscr{H}(x)$.  It is enough to prove that, if $f$ and $g$ are elements of $\widehat{\mathcal{O}}_{\cX,y}$ and $g$ is non-zero, then
  the valuation of $(f/g)^m-t$ in $\mathscr{H}(x)$ is uniformly bounded above, so that cannot approximate $t$ by $m$-th powers in $\mathscr{H}(x)$.
  We consider the morphism of $\kappa(y)$-algebras $\varphi:\widehat{\mathcal{O}}_{\cX,y}\to \kappa(y)\llbr s \rrbr$ defined by $z_i\mapsto \lambda_i s$ for $i=0,\ldots,r-1$ and
  $z_r\mapsto \lambda_r s^{\varepsilon}$, where $\varepsilon \in \{1,2\}$ is chosen such that $r-1+\varepsilon$ is not divisible by $m$, and the $\lambda_i$ are general elements in $\kappa(y)$ such that $v_x(g)\geq \ord_s(\varphi(g))w_{\min}$.
    Then $\ord_s(\varphi(t))=r-1+\varepsilon$ is not divisible by $m$, so that $$\ord_s(\varphi((f/g)^m-t))\leq r-1+\varepsilon.$$ Since $v_x(h)\leq \ord_s(\varphi(h))$ for every element $h$  of $\widehat{\mathcal{O}}_{\cX,y}$, we conclude that $$v_x((f/g)^m-t)\leq r+1+m/w_{\min}.$$
    \fi
 \end{proof}

\sss We denote by $\delta(X)$ the dimension of the essential skeleton $\Sk(X)$, and we call this number the {\em degeneracy index} of $X$. By definition, it is contained in the
range $\{0,\ldots,\dim(X)\}$. It follows from Proposition \ref{prop:bc} that $\delta$ is invariant under finite base change: we have $\delta(X)=\delta(X(d))$ for every $d>0$.
 The degeneracy index can be viewed as a measure for the degeneration of $X$ at $t=0$ (up to finite base change). For instance, $\delta(X)=0$ if $X$ has potential good reduction.
We will discuss some more examples in Section \ref{sec:examples}. If $\cX$ is an snc-model of $X$ over $R$ with $\cX_k=\sum_{i\in I}N_i E_i$ and $\mathrm{div}_{\cX}(\omega)=\sum_{i\in I}\nu_i E_i$, then the description of $\Sk(X)$ in \eqref{sss:ess} implies at once that $\delta(X)+1$ is the maximal cardinality of
 a subset $J$ of $I$ such that $\cap_{j\in J}E_j$ is non-empty and $\nu_j/N_j=\min(\omega)-1$ for every $j\in J$.

\subsection{The largest pole of the motivic zeta function}

\sss In this section, we will relate the essential skeleton of $X$ to the largest pole of the motivic zeta function $Z_{X,\omega}(T)$. The notion of pole requires some care because the ring $\gro$ is not a domain. We adopt the following definition. Let $q$ be a rational number and let $m$ be a nonnegative integer. We say that $Z_{X,\omega}(T)$
has a pole of order at most $m$ at $q$ if we can find a set $\mathcal{S}$ consisting of multisets in
$\Z \times \Z_>0$ such that:
\begin{enumerate}
\item  each multiset $S \in \mathcal{S}$ contains at most $m$ elements $(a, b)$ with $a/b = q$, and

\item $Z_{X,\omega}(T)$ belongs to the sub-$\gro[T]$-module of $\gro\llbr T\rrbr$ generated by
$$\left\{ \prod_{(a,b)\in S}\frac{1}{1-\LL^a T^b}\,|\,S\in \mathcal{S}\right\}.$$
\end{enumerate}
We say that $Z_{X,\omega}(T)$ has a pole of order $m$ at $q$ if it has a pole of order at most $m$,
but not of order at most $m-1$ (the latter condition is void for $m=0$). We say that $Z_{X,\omega}(T)$ has a pole at $q$ if it has a pole of positive order at $q$.
 This terminology is explained by the fact that, when considering motivic generating series of this type, one usually makes a formal substitution $T=\LL^{-s}$ and considers it as a series in the variable $s$.

\sss It is obvious from the expression in Theorem \ref{thm-snc} that every pole $q$ of $Z_{X,\omega}(T)$ satisfies $q\leq 1-\min(\omega)$ and that the order of a pole of $Z_{X,\omega}(T)$ is at most $\mathrm{dim}(X)+1$ (because $E_J$ is empty if $|J|>\mathrm{dim}(X)+1$).  It also follows immediately from Theorem \ref{thm-snc} that $Z_{X,\omega}(T)$ has a pole of order at most $\delta(X)+1$ at $1-\min(\omega)$. We will now show that this bound is sharp.

\begin{thm}\label{thm:largest}
Let $X$ be a geometrically connected, smooth and proper $K$-scheme with trivial canonical line bundle, and let $\omega$ be a volume form on $X$.
 Then $Z_{X,\omega}(T)$ has a pole of order $\delta(X)+1$ at $1-\min(\omega)$, and this is the largest pole of $Z_{X,\omega}(T)$.
\end{thm}
\begin{proof}
We only need to prove that the order of the pole at $1-\min(\omega)$ is not smaller than $\delta(X)+1$. This is done by a direct residue calculation on the expression for $Z_{X,\omega}(T)$ in Theorem \ref{thm-snc}. In order to avoid complications related to the presence of zero-divisors in $\gro$, we first specialize the coefficients of the zeta function to an integral domain. Denote by $\mathcal{M}_k=K_0(\Var_k)[\LL^{-1}]$ the localized Grothendieck ring of $k$-varieties without group action. The {\em Poincar\'e specialization}
$$P:\mathcal{M}_k \to \Z[u,u^{-1}]$$ is the unique ring morphism that sends $[Y]$ to the Poincar\'e polynomial
$$P(Y,u)=\sum_{i\geq 0}(-1)^i \mathrm{dim}H^i(Y(\C),\Q)u^i$$
for every smooth and proper $k$-scheme $Y$. The existence of such a morphism $P$ can be deduced from Hodge theory or weak factorization; see \cite[\S8]{Ni-tracevar}. For every $k$-scheme $Y$ of finite type, we will write $P(Y,u)$ for the image of $[Y]$ under $P$.
The image of $\LL=[\mathbb{P}^1_k]-1$ is equal to $u^2$.

 We write $1-\min(\omega)$ as $a/b$ for some integers $a,b$ with $b>0$.
Forgetting the $\widehat{\mu}$-action on $Z_{X,\omega}(T)$ and specializing the formula in Theorem \ref{thm-snc} by means of $P$, we obtain a series
$$Z'(T)=\sum_{\emptyset \neq J\subset I}P(\widetilde{E}^o_J,u)(u^2-1)^{|J|-1}\prod_{j\in J}\frac{u^{-2\nu_j}T^{N_j}}{1-u^{-2\nu_j}T^{N_j}}.$$
 It suffices to show that this series, viewed as a formal power series in $T$ over the field $\Q(u^{1/b})$, has a pole of order $\delta(X)+1$ at $T=u^{-2a/b}$.
  Evaluating $(1-u^{2a/b}T)^{\delta(X)+1}Z'(T)$ at $T=u^{-2a/b}$, we find
  a finite sum of expressions of the form
  $$P(Y,u)(u^2-1)^{\beta}\prod_{j=1}^r \frac{u^{-\alpha_j/b}}{1-u^{-\alpha_j/b}}$$
   where $Y$ is a $k$-scheme of finite type, $\beta$ is a nonnegative integer and $r$ and the $\alpha_j$ are positive integers. Here we used that $\mu_i/N_i \geq \min (\omega)-1= -a/b$ for every $i\in I$. Developing these terms as Laurent series in $u^{-1/b}$, we get a finite sum of Laurent series whose leading coefficients are all positive
   because, for every $k$-scheme of finite type $Y$, the polynomial $P(Y,u)$ has degree $2\mathrm{dim}(Y)$ and the coefficient of $u^{2\mathrm{dim}(Y)}$ is positive (it is the number of irreducible components of maximal dimension of $Y$ \cite[8.1]{Ni-tracevar}). Thus we find that
   the value of $(1-u^{2a/b}T)^{\delta(X)+1}Z'(T)$ at $T=u^{-2a/b}$ is non-zero. This concludes the proof.
\end{proof}

\begin{remark}
Theorem \ref{thm:largest} has a natural counterpart for Denef and Loeser's motivic zeta function of a hypersurface singularity.
Let $X$ be a connected smooth $k$-scheme, let $f:X\to \A^1_k$ be a dominant morphism and let $x$ be a closed point on $X$ such that $f(x)=0$.
 Let $h:Y\to X$ be a log resolution for $(X,\mathrm{div}(f))$ and write $\mathrm{div}(f\circ h)=\sum_{i\in I}N_i E_i$ and $K_{Y/X}=\sum_{i\in I}(\nu_i-1)E_i$.
  Denote by $\lct_x(f)$ the log canonical threshold of $f$ at $x$, that is
  $$\lct_x(f)=\min\{\frac{\nu_i}{N_i}\,|\,i\in I,\,E_i\cap h^{-1}(x)\}.$$
  Let $Z_{f,x}(T)\in \gro\llbr T\rrbr$ be the motivic zeta function of $f$ at $x$ (the fiber over $x$ of the zeta function $Z(T)$ in \cite[3.2.1]{DL-barc}). Then it follows from Denef and Loeser's formula \cite[3.3.1]{DL-barc} and a similar residue calculation as in the proof of Theorem \ref{thm:largest} that $-\lct_x(f)$ is the largest pole of
  $Z_{f,x}(T)$ and that its order is equal to the maximal cardinality of a subset $J$ of $I$ such that $(\cap_{j\in J}E_j)\cap h^{-1}(x)$ is non-empty and
  $\nu_j/N_j=\lct_x(f)$ for every $j\in J$. See \cite[3.5]{NiXu2}.
\end{remark}

\subsection{Relation with Hodge theory}
\sss
The aim of this section is to relate the invariants $\min(\omega)$ and $\delta(X)$ to the limit mixed Hodge structure associated with the $K$-scheme $X$.
  In particular, we will show that $\exp(-2\pi i \min(\omega))$ is an eigenvalue of the monodromy operator $\sigma$ on $H^n(X\times_K K^a,\Q)$, where $n$ is the dimension of $X$.
     Thus the pole $1-\min(\omega)$ of $Z_{X,\omega}(T)$ satisfies the Monodromy Property in Definition \ref{def:MP}. We announced a less general version of Theorem \ref{thm:hodge} in \cite[6.7]{HaNi-CY} without proof; a different, independent proof of Theorem \ref{thm:hodge} also appeared recently in \cite[Thm.A]{EFM} in the case where $X$ is defined over the field of germs of meromorphic functions at the origin of the complex plane.

\sss
 The limit mixed Hodge structure on the cohomology spaces $H^i(X\times_K K^a,\Q)$ was constructed by Stewart and Vologodsky in \cite[\S2.2]{StVo}, using logarithmic geometry (in particular the results on logarithmic de Rham cohomology in \cite{IKN}).
  It generalizes the classical construction of Steenbrink in the case where $X$ is defined over a complex punctured disk \cite{steenbrink}. We denote by $F^{\bullet}$ the Hodge filtration on $H^i(X\times_K K^a,\Q)$. Stewart and Vologodsky made the assumption that $X$ is projective, but this is only used to prove that $F^{\bullet}$ and the monodromy weight filtration $W_{\bullet}$ define a mixed Hodge structure, and not in the construction of $F^{\bullet}$ or the proof of the degeneration of the Hodge spectral sequence \cite[7.1]{IKN}.

  \begin{theorem}\label{thm:hodge}
  Let $X$ be a smooth, proper, geometrically connected $K$-scheme of dimension $n$ with trivial canonical line bundle, and let $\omega$ be a volume form on $X$.
   Then the semisimple part of the monodromy operator $\sigma$ acts on the one-dimensional complex vector space $$F^n H^n(X\times_K K^a,\C)$$ by multiplication with
   $\exp(-2\pi i \min(\omega))$. If $X$ is projective, then the action of $\sigma$ on $H^n(X\times_K K^a,\C)$ has a Jordan block with eigenvalue $\exp(-2\pi i \min(\omega))$ of rank at least $\delta(X)+1$.
    \end{theorem}
    \begin{proof}
    Let $\cX$ be an snc-model of $X$ over $R$. We denote by $S^{\dagger}$, resp.~$S(d)^{\dagger}$, the spectrum of $R$, resp.~$R(d)$, endowed with its standard log structure, for every $d>0$.
        We denote by $\loga{X}$ the
    log scheme obtained by endowing $\cX$ with the divisorial log structure induced by $\cX_k$. Then $\loga{X}$ is smooth over $S^{\dagger}$, and the logarithmic
    relative canonical line bundle $\omega_{\loga{X}/S^{\dagger}}$ is canonically isomorphic to $\omega_{\cX/R}(\cX_{k,\red}-\cX_k)$.

     Let $d$ be a sufficiently divisible positive integer and denote by $\loga{Y}$ the fiber product
     $$\loga{X}\times^{\mathrm{fs}}_{S^{\dagger}}S(d)^{\dagger}$$ in the category of fine and saturated (fs) log schemes. We denote by $\cY$ its underlying scheme; this is the normalization of $\cX\times_R R(d)$. Then the following properties hold:
     \begin{enumerate}
     \item the line bundle $\omega_{\loga{Y}/S(d)^{\dagger}}$ is the pullback of $\omega_{\loga{X}/S^{\dagger}}$ to $\cY$;
     \item the vector space $F^n H^n(X\times_K K^a,\C)$ is canonically isomorphic to $$H^0(\cY,\omega_{\loga{Y}/S(d)^{\dagger}})\otimes_{R(d)}k$$ by the degeneration of the Hodge spectral sequence;
     \item under this isomorphism, the action of the semisimple part of $\sigma$ on $F^n H^n(X\times_K K^a,\C)$ is induced by the Galois action of $\sigma_d=\exp(2\pi i/d)\in \mu_d(k)$ on $\cY$.
     \end{enumerate}
Now set $\alpha=\min(\omega)-1$.  Changing $d$ by a multiple, we may assume that $d\alpha$ is an integer.
            By the definition of $\min(\omega)$, the element $t^{-\alpha}\omega$ extends to a global section of $\omega_{\loga{Y}/S(d)^{\dagger}}$ that generates
            $\omega_{\loga{Y}/S(d)^{\dagger}}$ at the generic point of every component of $\cY_k$ that dominates a component $E_i$ of $\cX_k$ satisfying $\nu_i/N_i=\min(\omega)$ (notation as in \eqref{sss:snc}). Thus the reduction of $t^{-\alpha}\omega$ modulo $t^{1/d}$ generates $F^n H^n(X\times_K K^a,\C)$. Since $\omega$ is defined over $K$, $\sigma_d\in \mu_d(k)$ acts on $F^n H^n(X\times_K K^a,\C)$ (from the left) by multiplication with $\exp(-2\pi i \alpha)=\exp(-2\pi i \min(\omega))$. This proves the first part of the statement.

            Now assume that $X$ is projective; then we can choose $\cX$ to be projective over $R$.  To prove the statement on Jordan blocks,  it suffices to show that \begin{equation}\label{eq:weight}
            \mathrm{gr}^W_{n+\delta(X)} F^n H^n(X\times_K K^a,\C)\neq 0,
            \end{equation}
            because the weight filtration of the limit mixed Hodge structure on $H^n(X\times_K K^a,\Q)$ is the monodromy weight filtration associated with $\sigma$.
 We first reduce to the case where $X$ is defined over an algebraic curve, so that we can apply further results from Steenbrink's theory.
  The property \eqref{eq:weight} is invariant under finite extension of the base field $K$, so that we may assume that the special fiber $\cX_k$ is reduced, by the semistable reduction theorem.
   By construction, the limit mixed Hodge structure on $H^n(X\times_K K^a,\Q)$ only depends on the restriction of $\loga{X}$ over the standard log point $(\Spec k)^{\dagger}$.
   The same is true for the degeneracy index $\delta(X)$, by the proof of \cite[4.2.3]{NiXu}. In
    particular, they only depend on $\cX\times_R R/(t^2)$, by Illusie's result in \cite[A.4]{nakayama}. Thus, by a standard application of spreading out and Greenberg approximation, we may assume that $\cX$ is defined over an algebraic $k$-curve -- see for instance \cite[5.1.2]{MuNi}.

    So we change notation and denote by $\cX$ a regular projective flat scheme over a smooth $k$-curve $C$, by $s$ a closed point on $C$ and by $t$ a uniformizer in $\mathcal{O}_{C,s}$ such that $X\cong \cX\times_C \Spec K$ and  $\cX_s$ is a reduced divisor with strict normal crossings. We write $f:\cX\to C$ for the structural morphism. Then the mixed Hodge structure on $H^n(X\times_K K^a,\Q)$ coincides with Steenbrink's limit mixed Hodge structure on the complex analytic nearby cohomology $\mathbb{H}^n(\cX_s(\C),R\psi_{f}(\Q))$.
     The complex component of Steenbrink's limit mixed Hodge structure is defined in \cite[4.17]{steenbrink} by means of a cohomological mixed Hodge complex  $(A^{\bullet}_{\C},F^{\bullet},W_{\bullet})$ of sheaves of complex vector spaces on $\cX_s$. Here $A^{\bullet}_{\C}$ is the simple complex associated with a double complex $A^{\bullet \bullet}_{\C}$ that satisfies
     $$A^{pq}_{\C}=\Omega^{p+q+1}_{\cX/k}(\log \cX_s)/W_q\Omega^{p+q+1}_{\cX/k}(\log \cX_s)$$
for $p,q\geq 0$,     where $W_\bullet$ is the usual weight filtration on  $\Omega^{p+q+1}_{\cX/k}(\log \cX_s)$.
      Shrinking $C$ around $s$, we may assume that $f_*\Omega^n_{\cX/C}(\log \cX_s)$ is free of rank one; let $\phi$ be a generator.
      Then, by the description of the essential skeleton $\Sk(X)$ in \eqref{sss:ess}, we know that $\phi\wedge dt/t$ defines a global section of
      $$W_{\delta(X)}\Omega^{n+1}_{\cX/k}(\log \cX_s)\otimes_{\mathcal{O}_{\cX}}\mathcal{O}_{\cX_s}$$
      that does not lie in
      $$W_{\delta(X)-1}\Omega^{n+1}_{\cX/k}(\log \cX_s)\otimes_{\mathcal{O}_{\cX}}\mathcal{O}_{\cX_s}.$$
      This implies that $$H^0(\cX_s,\mathrm{gr}_i^WF^nA^{\bullet}_{\C})=H^0(\cX_s,\mathrm{gr}_i^W\Omega^{n+1}_{\cX/k}(\log \cX_s))$$
      is non-zero if and only if $i=\delta(X)$. Thus by applying the exact functor $F^n=\mathrm{gr}_F^n$ to the weight spectral sequence in \cite[4.20]{steenbrink},
       we find that
       $$\mathrm{gr}^W_{n+\delta(X)}F^n \mathbb{H}^n(\cX_s(\C),R\psi_{f}(\C)) \cong H^0(\cX_s,\mathrm{gr}_{\delta(X)}^WF^nA^{\bullet}_{\C})\neq 0.$$
       This concludes the proof.
    \end{proof}

\subsection{Poles of maximal order}
\sss It follows immediately from Theorem \ref{thm-snc} that each pole of $Z_{X,\omega}(T)$ has order at most $\dim(X)+1$. For Denef and Loeser's motivic zeta functions of hypersurface singularities, it was shown by Xu and the second author in \cite{NiXu2} that the only possible pole of order $\dim(X)+1$ is the largest pole, i.e., minus the log canonical threshold of $f$ at $x$. This property had been conjectured by Veys in \cite{LaVe}. The results in \cite{NiXu2} imply that the analogous property holds for Calabi-Yau varieties.

\begin{thm}\label{thm:maxorder}
Let $X$ be a geometrically connected, smooth and proper $K$-scheme with trivial canonical line bundle, and let $\omega$ be a volume form on $X$.
 Then $Z_{X,\omega}(T)$ has a pole of order $\dim(X)+1$ if and only if $X$ has maximal degeneracy index; that is, $\delta(X)=\dim(X)$.
 Moreover, if $s$ is a pole of
order $\dim(X)+1$ of $Z_{X,\omega}(T)$, then $s=1-\min(\omega)$.
\end{thm}
\begin{proof}
 In view of theorem \ref{thm:largest}, it is enough to prove the second part of the statement: if $s$ is a pole of
order $\dim(X)+1$ of $Z_{X,\omega}(T)$, then $s=1-\min(\omega)$. Let $\cX$ be an snc-model of $X$. We adopt the notations from \eqref{sss:snc}, writing $\cX_k=\sum_{i\in I}N_i E_i$
 and $\mathrm{div}_{\cX}(\omega)=\sum_{i\in I}\nu_i E_i$. It follows from the explicit formula for the motivic zeta function in Theorem \ref{thm-snc} that a rational number $s$ is a pole of $Z_{X,\omega}(T)$ of order $\dim(X)+1$ if and only if there exists a subset $J$ of $I$ of cardinality $\dim(X)+1$ such that $E_J$ is non-empty and $\nu_j/N_j=-s$ for every $j\in J$. Then each connected component of $E_J$ corresponds to a face of $\Sk(\cX)$ of dimension $\dim(X)+1$ on which the weight function $\wt_{\omega}$ is constant with value $-s$. Now Theorem 5.4 in \cite{NiXu2} implies that $s$ equals $1-\min(\omega)$.
\end{proof}

\sss If $\delta(X)<\dim(X)$ then $Z_{X,\omega}(T)$ may very well have more than one pole of order $\delta(X)+1$. This happens, for instance, in Example \ref{ex:nokulikov}, where $\delta(X)=0$ and the zeta function has two poles of order one.
\section{Abelian varieties}\label{sec:abelian}
In the case where $X=A$ is an abelian variety over $K$, the motivic zeta function $Z_{A,\omega}(T)$ was studied in depth in \cite{HaNi} forgetting the $\widehat{\mu}$-action.
 We will now explain how these results can be upgraded to take the $\widehat{\mu}$-action into account. In particular, we will show that the Monodromy Property holds for abelian varieties.
\subsection{Auxiliary results}\label{ss:abelian-aux}
\sss Let $G$ be a smooth commutative group scheme locally of finite type over $k$. We denote by $\pi_0(G)$ the group of connected components of $G$. The identity component $G^0$ has a canonical Chevalley decomposition
$$0\to L\to G^o\to B\to 0$$
where $L$ is a smooth connected affine group scheme over $k$ and $B$ is an abelian variety. The group scheme $L$ splits canonically as $L\cong U\times_k T$ where $U$ is a unipotent group scheme and $T$ is a torus. The dimensions of $T$ and $U$ are called the toric and unipotent ranks of $G$, respectively. We will denote by $G^{\sharp}$ the group scheme $G/L$ over $k$; this is an extension of $\pi_0(G)$ by the abelian variety $B$. Note that $G^{\sharp}$ is functorial in $G$ because there are no non-trivial morphisms from $L$ to an abelian variety. Applying this functor commutes with taking identity components, and the projection $G\to G^{\sharp}$ induces an isomorphism of component groups
$\pi_0(G)\to \pi_0(G^{\sharp})$.

\begin{prop}\label{prop:chevalley}
Assume that
$G$ is of finite type over $k$ and carries a good action of the profinite group $\widehat{\mu}$, such that $\widehat{\mu}$ acts trivially on the torus $T$.
Denote by $\tau$ and $u$ the toric and unipotent ranks of $G$, respectively.
 Then we have $$[G]=[G^{\sharp}]\LL^u(\LL-1)^\tau$$ in $K_0^{\widehat{\mu}}(\Var_k)$.
\end{prop}
\begin{proof}
Since $k$ has characteristic zero, $U$ is canonically isomorphic to the vector group scheme associated with $\Lie(U)$ {\em via} the exponential map.
 Thus $G$ is an affine bundle over $G/U$ and the action of $\widehat{\mu}$ on $G$ is affine.
Hence, $[G]=[G/U]\LL^u$ in $K_0^{\widehat{\mu}}(\Var_k)$ because we have trivialized affine actions in the definition of the equivariant Grothendieck ring.
 Choosing an isomorphism $T\cong \mathbb{G}^\tau_m$ we can also view $G/U$ as a product of $t$ punctured line bundles over $G^{\sharp}$, where ``punctured'' means that we remove the zero sections. Since the action of $\widehat{\mu}$ on $T$ is trivial, the group $\widehat{\mu}$ acts linearly on each punctured line bundle. The triviality of affine actions and the scissor relations now imply that $[G/U]=(\LL-1)^\tau[G^{\sharp}]$.
\end{proof}

\sss Now let $E$ be a semi-abelian variety over $K$, that is, an extension of an abelian variety by a torus. Let $\cE$ be the N\'eron lft-model of $E$ over $R$ and denote by $\cE^o_k$ the identity component of its special fiber (the prefix lft indicates that $\cE$ is locally of finite type over $R$, rather than of finite type, if $E$ contains a non-trivial split subtorus).
  We say that $E$ has semi-abelian reduction if $\cE^o_k$ is semi-abelian.
  For every positive integer $d$, we denote by $\cE(d)$ the N\'eron model of $E(d)=E\times_K K(d)$. It carries a natural action of the Galois group $\mu_d=\mathrm{Gal}(K(d)/K)$.
  One can deduce from Grothendieck's Semi-Stable Reduction Theorem for abelian varieties that there exists
 a positive integer $e$ such that $E(e)$ has semi-abelian reduction -- see \cite[III.3.6.4]{HaNi-book}.
 If $A$ is an abelian $K$-variety with N\'eron model $\cA$, then
 the toric and unipotent ranks of $A$ are, by definition, the toric and unipotent ranks of $\cA_k^o$. We denote them by $t(A)$ and $u(A)$, respectively.
 If $e$ is a positive integer such that $A(e)$ has semi-abelian reduction, then the toric rank of $A(e)$ does not depend on the choice of $e$. It is called the {\em potential toric rank} of $A$ and denoted by $t_{\mathrm{pot}}(A)$. We refer to \cite{HaNi-book} for further background.

\begin{lemma}\label{lemm:trivaction}
Let $E$ be a semi-abelian $K$-variety whose abelian part has potential good reduction (that is, the potential toric rank is zero).
 Let $d,e$ be positive integers such that $E(e)$ has semi-abelian reduction and $d$ is prime to $e$.
  Then the Galois action of $\mu_d$ on $\cE(d)^{\sharp}_k$ is trivial. If $E$ has no non-trivial split subtorus,
   then the base change morphism
   $$\cE^{\sharp}_k\to \cE(d)^{\sharp}_k$$ is an isomorphism.
\end{lemma}
\begin{proof}
 We first consider the case where $E$ has no non-trivial split subtorus. Then $E(d)$ has no non-trivial split subtorus either, because the toric part of $E$ splits over $K(e)$ and $d$ is prime to $e$.
  The toric rank of  $\cE(d)^o_k$ is zero by \cite[3.13]{HaNi-comp}, so that $\cE(d)^o_k$ is an extension of an abelian variety by a unipotent group.
   The theory of Edixhoven's filtration on N\'eron models \cite{edix}, generalized to semi-abelian varieties in \cite[\S4]{HaNi}, provides a canonical identification of
   $\cE(d)^{\sharp}_k$ with the fixed locus of $\cE(de)_k$ under the action of the Galois group $\mathrm{Gal}(K(de)/K(d))$ -- see in particular \cite[4.8]{HaNi}.
    In the same way, we can identify $\cE_k^{\sharp}$ with the fixed locus of $\cE(de)_k$ under the action of $\mathrm{Gal}(K(de)/K)$. It follows that
 $\cE_k^{\sharp}$ is canonically isomorphic to the fixed locus of $\cE(d)^{\sharp}_k$ under the action of $\mu_d=\mathrm{Gal}(K(d)/K)$.

 The base change morphism
 $$\cE(e)\times_{R(e)}R(de)\to \cE(de)$$ is an open immersion because $\cE(e)_k^o$ is semi-abelian \cite[IX.3.3]{sga7-1}. Thus $\mathrm{Gal}(K(de)/K(e))$ acts trivially on
 the identity component of $\cE(de)_k$. Hence, $\mu_d$ acts trivially on the identity component of $\cE(d)^{\sharp}_k$. We also know that
 the morphism of groups of connected components associated with the base change morphism $\cE_k\to \cE(d)_k$ is an isomorphism, by \cite[V.3.3.11]{HaNi-book}.
 Since the projection $\cE_k\to \cE_k^{\sharp}$ induces an isomorphism on component groups and the same holds for $\cE(d)$, we find that the base change morphism
 $$\cE^{\sharp}_k\to \cE(d)^{\sharp}_k$$ is an isomorphism, which means that $\mu_d$ acts trivially on $\cE(d)^{\sharp}_k$.

 Now, we deduce the general case.
 Let $T$ be the maximal split subtorus of $E$ and set $E'=E/T$. We denote by $\cE$, $\mathscr{E}'$ and $\mathscr{T}$ the N\'eron lft-models of $E$, $E'$ and $T$, respectively. Then the natural sequence
$$0\to \cT\to \cE\to \cE'\to 0$$ is exact, by the proof of \cite[10.1.7]{BLR}. Since the toric part of $E$ splits over $K(e)$ and $d$ is prime to $e$, $T(d)$ is still the maximal split subtorus of $E(d)$.
   We denote by $\cE(d)$, $\mathscr{E}'(d)$ and $\mathscr{T}(d)$ the N\'eron lft-models of $E(d)$, $E'(d)$ and $T(d)$, respectively.

The group scheme $\cT(d)^{\sharp}_k$ is simply the group of connected components of $\cT(d)_k$. It is canonically isomorphic to $$\Hom_{\Z}(X(T),K(d)^{\ast}/R(d)^{\ast})$$ where $X(T)$ denotes the character group of
   $T$ -- see for instance \cite[III.3.4]{HaNi-book}. In particular, the action of $\mu_d$ on this group is trivial.
    The construction of the N\'eron lft-model $\cE(d)$ in the proof of \cite[10.1.7]{BLR} implies that the morphism of component groups induced by $\cT(d)_k\to \cE(d)_k$ is injective. It follows that the sequence of group schemes
   $$0\to \cT(d)^{\sharp}_k \to \cE(d)^{\sharp}_k\to \cE'(d)_k^{\sharp}\to 0$$ is still exact.
 We have already proven that
 $\mu_d$ acts trivially on $\cE'(d)_k^{\sharp}$. Then it also acts trivially on
 $$\cE(d)^{\sharp}_k\times_{\cE'(d)_k^{\sharp}} (\cE'(d)_k^{\sharp})^o$$
  by the triviality of the action on $\cT(d)^{\sharp}_k$. Moreover, as we have already recalled, the base change morphism of component groups
  $$\pi_0(\cE'_k)\to \pi_0(\cE'(d)_k)$$ is an isomorphism. This means that we can lift any
  point $c$ of $\pi_0(\cE'(d)_k)\cong \pi_0(\cE'(d)^{\sharp}_k)$ to a point of $\cE'(R)\subset \cE'(d)(R(d))$.
  This point lifts, at its turn, to a point $x_c$ in $\cE(R)$ because $H^1(K,T)=0$. Multiplication by $x_c$ now defines a $\mu_d$-equivariant isomorphism between
  $$\cE(d)^{\sharp}_k\times_{\cE'(d)_k^{\sharp}} (\cE'(d)_k^{\sharp})^o=\cE(d)^{\sharp}_k\times_{\pi_0(\cE'(d)_k^{\sharp})}\{1\}$$
  and $$\cE(d)^{\sharp}_k\times_{\pi_0(\cE'(d)_k^{\sharp})}\{c\}.$$
     Hence, $\mu_d$ acts trivially on the latter scheme, for all $c$, and thus also on
     $\cE(d)^{\sharp}_k$.
\end{proof}

\begin{prop}\label{prop:trivaction}
Let $A$ be an abelian $K$-variety. Let $d,e$ be positive integers such that $A(e)$ has semi-abelian reduction and $d$ is prime to $e$. Then the group $\mu_d$ acts trivially on the maximal subtorus of $\cA(d)^o_k$, and on $\cA(d)^{\sharp}_k$.
  Moreover, the natural base change morphism
  $$\cA\times_R R(d)\to \cA(d)$$ induces an open immersion
  $$\cA_k^{\sharp}\to \cA(d)^{\sharp}_k.$$
 \end{prop}
\begin{proof}
We consider the non-archimedean uniformization $E^{\an}\to A^{\an}$ of $A$, where $E$ is an extension of an abelian variety $B$ with potential good reduction by a $K$-torus.
 Let $T'$ be the maximal split subtorus of $E$, and set $E'=E/T'$.
  We denote by $\cE$, $\cE'$, $\cT$ and $\cA$ the N\'eron lft-models of $E$, $E'$, $T$ and $A$, respectively. Then
   the natural sequence
$$0\to \cT\to \cE\to \cE'\to 0$$ is exact, by the proof of \cite[10.1.7]{BLR}.
   Moreover, there exists a natural surjective morphism of group schemes $\cE_k\to \cA_k$ that induces an isomorphism on the identity components \cite[2.3]{BX}.
  Under this isomorphism, the maximal subtorus of $\cA_k^o$ is precisely $\cT_k^o$, because the toric rank of $(\cE'_k)^o$ is zero by \cite[3.13]{HaNi-comp}.

Since the toric part of $E$ splits over $K(e)$ and $d$ is prime to $e$, $T(d)$ is the maximal split subtorus of $E(d)$.
   We denote by $\cE(d)$, $\mathscr{E}'(d)$ and $\mathscr{T}(d)$ the N\'eron lft-models of $E(d)$, $E'(d)$ and $T'(d)$, respectively.
 Then the results of the previous alinea also apply to these objects. In particular, we can again identify the maximal subtorus of $\cA(d)_k^o$ with $\cT(d)^o_k$.
  But $\mu_d$ acts trivially on $\cT(d)_k^o$ because $T$ is split over $K$, and it follows that $\mu_d$ acts trivally on the maximal subtorus of $\cA(d)_k^o$.
  Since $\cA(d)^{\sharp}_k$ is a quotient of $\cE(d)^{\sharp}_k$, and $\mu_d$ acts trivially on  $\cE(d)^{\sharp}_k$ by Lemma \ref{lemm:trivaction},
  we see that $\mu_d$ acts trivially on $\cA(d)^{\sharp}_k$.

It remains to prove that the morphism
     $$\cA_k^{\sharp}\to \cA(d)^{\sharp}_k$$ is an open immersion. The groups of connected components of
     $\cA_k$ and $\cA_k^{\sharp}$ are canonically isomorphic, and the same holds for $\cA(d)_k$ and $\cA(d)_k^{\sharp}$. The
     morphism of component groups induced by the base change morphism  $$\cA\times_R R(d)\to \cA(d)$$ is injective, by \cite[5.5]{HaNi-comp}.
     Thus we only need to show that $$\cA_k^{\sharp}\to \cA(d)^{\sharp}_k$$ induces an isomorphism on the identity components; we can identify this morphism on identity components with the base change morphism $$\cE^{\sharp,o}_k\to \cE(d)^{\sharp,o}_k.$$
 However, the natural morphism $\cE^o_k\to (\cE'_k)^o$ is surjective and its kernel is a torus, so that it induces an isomorphism $(\cE^\sharp_k)^{o}\to (\cE'_k)^{\sharp,o}$.
  The analogous statement holds for $\cE(d)$ and $\cE'(d)$. We have already proven in Lemma \ref{lemm:trivaction} that the base change morphism
  $$(\cE')^{\sharp}_k\to \cE'(d)^{\sharp}_k$$ is an isomorphism. Therefore, the base change morphism
  $$\cA^{\sharp,o}_k\to \cA(d)^{\sharp,o}_k$$ is an isomorphism, as well.
\end{proof}

\subsection{Proof of the Monodromy Property}
\sss To formulate our main result for abelian varieties, we need to recall one more important invariant. Let $A$ be an abelian variety over $K$ and let $e$ be a positive integer such that $A(e)$ has semi-abelian reduction.
  Then there exists a canonical base change morphism
 $h:\cA\times_R R(e)\to \cA(e)$ that induces an injective morphism  $\Lie(h):\Lie(\cA)\otimes_R R(e)\to \Lie(\cA(e))$ of free $R(e)$-modules of rank $g=\dim(A)$. Chai's {\em base change conductor} of $A$ is the non-negative rational number
 $$c(A)=\frac{1}{e}\mathrm{length}_{R(e)}\coker(\Lie(h)).$$
 It is a measure for the defect of semi-abelian reduction of $A$; in particular, $c(A)=0$ if and only if $A$ has semi-abelian reduction.
 The invariant $c(A)$ does not depend on the choice of $e$. We refer to \cite{HaNi-book} for background.
 A volume form $\omega$ on $A$ is called {\em distinguished} if it extends to a relative volume form on $\cA$ over $R$.

\begin{thm}\label{thm:abelian}
Let $A$ be an abelian $K$-variety and let $\omega$ be a volume form on $A$. Then the motivic zeta function
$Z_{A,\omega}(T)$ belongs to the subring
 $$\gro \left[T,\frac{1}{1-\LL^a T^b}\right]_{(a,b)\in \Z\times \Z_{>0},\,a/b=1-\min(\omega)}$$
 of $\gro\llbr T\rrbr$. It has a unique pole at $s=1-\min(\omega)$, whose order equals $1+t_{\mathrm{pot}}(A)$.
 In particular, $A$ satisfies the Monodromy Property. If $\omega$ is distinguished, then $c(A)=1-\min(\omega)$.
\end{thm}
\begin{proof}
Rescaling $\omega$ by a unit in $K$, we may assume that it is distinguished, in view of equation \eqref{eq:rescale}.
If we forget the $\widehat{\mu}$-action, Theorem \ref{thm:abelian} was proven in \cite[8.6]{HaNi}. We will explain what needs to be changed in the proof to
 keep track of the $\widehat{\mu}$-action. For every positive integer $d$, we set $A(d)=A\times_K K(d)$ and we denote by $\cA(d)$ the N\'eron model of $A(d)$ over $R(d)$. By the universal property of the N\'eron model, the Galois action on $A(d)$ extends to $\cA(d)$ so that $\cA(d)$ is, in particular, an equivariant weak N\'eron model, which we can use to compute the motivic integral of $\omega(d)$ on $A(d)$. The volume form $\omega(d)$ has the same order along each of the connected components of $\cA(d)_k$, by translation invariance. We define $\mathrm{ord}_A(d)$ to be the opposite of this order. The integer $\mathrm{ord}_A(d)$ does not depend on the choice of $\omega$, since $\omega$ is determined up to multiplication with a unit in $R$.
  Then, by the definition of the motivic integral, we have
  $$\int_{X(d)}|\omega(d)|=[\mathscr{A}(d)_k]\LL^{\mathrm{ord}_A(d)}$$ in $\gro$.
  The proof in \cite[8.6]{HaNi} revolved around the following two key facts. Let $e$ be the degree of the minimal extension of $K$ where $A$ acquires semi-abelian reduction.
  \begin{enumerate}
\item We have $\mathrm{ord}_A(m+qe)=\mathrm{ord}_A(m)+c(A)eq$ for all positive integers $m$ and $q$.
\item \label{it:abproof} We have $[\mathscr{A}(md)_k]=d^{t(A(m))}[\mathscr{A}(m)_k]$ in $K_0(\Var_k)$ for all positive integers $d$ and $m$ such that $d$ is prime to   $e'=e/\gcd(m,e)$.
\end{enumerate}
It suffices to show that the equality in \eqref{it:abproof} remains valid in the equivariant Grothendieck ring $K^{\widehat{\mu}}_0(\Var_k)$.
 The group schemes $\cA(m)^o_k$ and $\cA(md)^o_k$ have the same toric rank, by  \cite[4.2]{HaNi-comp}. Moreover, it follows from Proposition \ref{prop:trivaction} that the base change morphism
 $$\cA(m)^\sharp_k\to \cA(md)^\sharp_k$$ is an open immersion, and that the $\widehat{\mu}$-action on $\cA(md)^\sharp_k$ factors through $\mu_m$.
  In particular, the identity components of  $\cA(m)^\sharp_k$ and $\cA(md)^\sharp_k$ have the same dimension, so that $\cA(m)^o_k$ and $\cA(md)^o_k$ have the same unipotent rank.
  Proposition \ref{prop:chevalley} now implies that
 $[\mathscr{A}(md)_k]=C[\mathscr{A}(m)_k]$ in $K^{\widehat{\mu}}_0(\Var_k)$, where $C$ is the order of the cokernel of
  $\cA(m)^\sharp_k\to \cA(md)^\sharp_k$. It follows from \cite[5.7]{HaNi-comp} that $C=d^{t(A(m))}$.
\end{proof}

\begin{cor}\label{cor:comparab}
The degeneracy index $\delta(A)$ of $A$ is equal to the potential toric rank $t_{\mathrm{pot}}(A)$.
\end{cor}
\begin{proof}
This follows immediately from Theorems \ref{thm:largest} and \ref{thm:abelian}.
\end{proof}

\begin{rema}
In \cite{HaNi}, we considered a more general set-up: $R$ was allowed to be any strictly Henselian discrete valuation ring, and $A$ was assumed to be a tamely ramified abelian variety over $K$. Let $k$ be the residue field of $R$ and denote by $p$ the characteristic exponent of $k$.
 Then, for every positive integer $d$ prime to $p$, the field $K=\mathrm{Frac}(R)$ still has a unique degree $d$ extension in some fixed separable closure of $K$, so that we can
 define $A(d)$, $\cA(d)$ and $\mathrm{ord}_A(d)$ as before.
We define the motivic zeta function $Z_{A,\omega}(T)$ by
$$Z_{A,\omega}(T)=\sum_{d>0,\,(d,p)=1}[\cA(d)_k]\LL^{-\mathrm{ord}_{A}(d)}T^d\quad \in \gro\llbr T\rrbr.$$
If we forget the $\widehat{\mu}$-action, this is precisely the zeta function that was considered in \cite{HaNi}. Theorem \ref{thm:abelian} remains valid in this more general case, and the proof carries over {\em verbatim}, except for one point: it is no longer true that every smooth connected commutative unipotent group scheme over $k$ is isomorphic to a power of the additive group. Thus in Proposition \ref{prop:chevalley}, we need to assume that $U$ has a $\widehat{\mu}$-equivariant filtration
such that each successive quotient is isomorphic to $\mathbb{G}^m_a$ for some $m>0$. If $U$ is the unipotent radical of $\cA(d)^o_k$ for some $d>0$ prime to $p$,  such a filtration
is provided by Edixhoven's theory \cite{edix}, and this is the only case we need.
\end{rema}

\subsection{The essential skeleton of an abelian variety}
\sss The equality $t_{\mathrm{pot}}(A)=\delta(A)$ in Corollary \ref{cor:comparab} can also be proven in a more direct way. By definition, $\delta(A)$ is the dimension of the essential skeleton $\Sk(A)$ of $A$.
In \cite[\S6.5]{berkbook}, Berkovich has given a different construction of a skeleton of $A$, denoted by $\Delta(A)$, whose dimension equals $t_{\mathrm{pot}}(A)$. It is a canonical subspace of the analytification $A^{\an}$. If $A$ has semi-abelian reduction, then $\Delta(A)$ is homeomorphic to a real torus
$(S^1)^{t(A)}$.

 \begin{prop}\label{prop:absk}
 Let $A$ be an abelian $K$-variety. Then the essential skeleton $\Sk(A)$ coincides with Berkovich's skeleton $\Delta(A)$ of $A$.
 \end{prop}
 \begin{proof}
 We will present Berkovich's construction in a slightly different (but equivalent) form, using the theory of non-archimedean uniformization. It suffices to consider the case where $A$ has semi-stable reduction, because both $\Sk(A)$ and $\Delta(A)$ are compatible with base change: for the essential skeleton this follows from Proposition \ref{prop:bc}, and for the Berkovich skeleton this is part of the construction.

 First, we recall the construction of the skeleton $\Delta(T)$ of a split algebraic torus $T$ over a complete non-archimedean field $F$.
 Denote by $M$ and $N=M^{\vee}$ the character and cocharacter lattice of $T$, respectively. Then there exists a canonical embedding of $N\otimes_{\Z}\R$ into $T^{\an}$, whose image is called the skeleton of $T$ and denoted by $\Delta(T)$. Under this embedding, a morphism $\varphi:M\to \R$ corresponds to the Gauss point of the poly-annulus in $T^{\an}$ defined by the equations $|m|=\exp(-\varphi(m))$ for all $m\in M$.

  Let $$0\to L\to E^{\an}\to A^{\an}\to 0$$ be the non-archimedean uniformization of $A$ \cite{BL1,BL2}. Here $E$ is an extension of an abelian $K$-variety $B$ with good reduction by a split $K$-torus $T$, and $L$ is a lattice of rank $\dim(T)$ in $E(K)$.
     Denote by $M$ and $N=M^{\vee}$ the character and cocharacter lattice of $T$, respectively.
   Let $\omega$ be a volume form on $A$; we denote its pullback to $E^{\an}$ by $\omega'$.
  The essential skeleton $\Sk(B)$ consists of a unique point $\xi$, namely, the divisorial point associated with the special fiber of the N\'eron model $\cB$ of $B$.
   We choose an open formal subscheme $\mU$ of the $t$-adic completion $\widehat{\cB}$ that contains the unit section of $\widehat{\cB}$ and such that the torus fibration $E\to B$ is trivial over the generic fiber of $\mU$ \cite[4.2]{BX}. Let $\sigma$ be a section of this fibration over $\mU_\eta$ that passes through the identity element of $E$.
   The fiber of $E$ over $\xi$ is a $T$-torsor over the completed residue field $\mathscr{H}(\xi)$, which we denote by $T_{\xi}$.
    We trivialize this torsor by choosing the rational point $x=\sigma(\xi)$ on $T_{\xi}$ as the identity in $T_\xi$. Then the embedding
    $N\otimes_{\Z}\R\to T_\xi^{\an}$ does not depend on the choice of $\sigma$. By means of the embedding $T^{\an}_\xi\to E^{\an}$, we can view the skeleton $\Delta(T_\xi)$ as a subspace of $E^{\an}$.

  The embedding $\Delta(T_\xi)\to E^{\an}$ has a canonical retraction: the tropicalization map $E^{\an}\to \Delta(T_\xi)$. It is completely characterized by the following properties:
 the restriction to $T^{\an}$ is the usual tropicalization map $$T^{\an}\to \Hom_{\Z}(M,\R):x\mapsto (m\mapsto -\ln |m(x)|),$$
  and for every connected analytic domain $U$ in $B^{\an}$ and every section $\sigma:U\to E^{\an}$ of the morphism $E^{\an}\to B^{\an}$, the map $\mathrm{trop}\circ \sigma$ is constant. The map $\mathrm{trop}$
  sends the lattice $L$ isomorphically onto
  a complete lattice in the $\R$-vector space $\Delta(T_\xi)$. Now the skeleton $\Delta(A)$ is the image of $\Delta(T_\xi)$ under the morphism $E^{\an}\to A^{\an}$, and the map
  $\Delta(T_\xi)\to \Delta(A)$ factors through a homeomorphism $\Delta(T_\xi)/\mathrm{trop}(L)\to \Delta(A)$.

  To complete the argument, it will be convenient to use Temkin's generalization of the weight function in \cite{temkin}. Temkin defined a metric $\|\cdot\|$ on the
  canonical line bundle of any quasi-smooth $K$-analytic space. By \cite[8.3.4]{temkin}, we have an equality $\wt_{\phi}=1-\ln \|\phi\|$ of functions on $Y^{\an}$ for every smooth and proper $K$-scheme $Y$ and every volume form $\phi$ on $Y$.
   Thus it is natural to define the weight function of a volume form $\phi$ on any quasi-smooth $K$-analytic space $Z$ by
  $$\wt_{\phi}:Z\to \R\cup \{+\infty\}:z\mapsto 1-\ln \|\phi(z)\|.$$
   Let $\omega'$ be the pullback of $\omega$ to $E^{\an}$. Since the morphism $h:E^{\an}\to A^{\an}$ is locally an open immersion, the local analytic nature of Temkin's construction implies that
   $\wt_{\omega'}=\wt_{\omega}\circ h$. Hence, it suffices to show that $\wt_{\omega'}$ reaches its minimal value precisely on $\Delta(T_\xi)$.
   Since $\omega'$ is translation-invariant, we can write it as $\omega'_B\otimes \omega'_T$ where $\omega'_B$ and $\omega'_T$ are translation invariant volume forms on $B$ and $T$, respectively. Then $\wt_{\omega'_B}$ reaches its minimal value exactly at $\xi$ and, by \cite[8.2.2]{temkin}, the function $\wt_{\omega'_T}$ reaches its minimal value exactly on the skeleton $\Delta(T)\subset T^{\an}$.
      Since the morphism $E\to B$ is a locally trivial fibration, it now easily follows from the definition of Temkin's metric that
      $\wt_{\omega'}$ reaches its minimal value precisely on $\Delta(T_\xi)$.
    \end{proof}
\begin{cor}
 Let $A$ be an abelian $K$-variety with semi-abelian reduction. Then the essential skeleton $\Sk(A)$ is homeomorphic to a real $t(A)$-dimensional torus $(S^1)^{t(A)}$.
\end{cor}
\begin{proof}
The Berkovich skeleton $\Delta(A)$ is homeomorphic to $(S^1)^{t(A)}$, so that the result follows from Proposition \ref{prop:absk}.
\end{proof}

\section{Equivariant Kulikov models}\label{sec:kulikov}
\subsection{Definitions}
\sss In Theorem \ref{thm:abelian}, we have shown that, when $X$ is an abelian variety over $K$ and $\omega$ is a volume form on $X$, the motivic zeta function $Z_{X,\omega}(T)$ has a unique pole. We will now extend this result to Calabi-Yau varieties $X$ that have a special kind of semistable model over a finite extension of $K$; we will call such models {\em equivariant Kulikov models} in analogy with the Kulikov classification of semistable degenerations of $K3$ surfaces (see Section \ref{sec:examples}). The unique pole is then equal to $1-\min(\omega)$, by Theorem \ref{thm:largest}, and it follows from Theorem \ref{thm:hodge} that $X$ satisfies the Monodromy Property in Definition \ref{def:MP}.

\begin{definition}
Let $X$ be a geometrically connected, smooth and proper $K$-scheme with trivial canonical line bundle, and let $d$ be a positive integer. A {\em Kulikov model} for $X$ over $R(d)$ is
 a regular flat proper algebraic space $\cY$ over $R(d)$, endowed with an isomorphism of $K(d)$-schemes $\cY_{K(d)}\to X\times_K K(d)$, such that the special fiber $\cY_k$ is a divisor with normal crossings, and the logarithmic relative canonical line bundle $\omega_{\cY/R(d)}(\cY_{k,\red}-\cY_k)$ is trivial. We say that the Kulikov model $\cY$ is {\em equivariant}
  if the Galois action of $\mu_d$ on $X\times_K K(d)$ extends to an action on $\cY$.
\end{definition}

\begin{example}
Assume that $X$ has potential Galois-equivariant good reduction, i.e., there exist a positive integer $d$ and a smooth and proper $R(d)$-model $\cY$ of $X\times_K K(d)$ such that the Galois action of $\mu_d$ on $X\times_K K(d)$ extends to $\cY$. Then $\cY$ is an equivariant Kulikov model for $X$.
\end{example}

\begin{remark} If $X$ has an equivariant Kulikov model $\cY$ over $R(d)$, for some $d>0$, then we can also find an equivariant Kulikov model $\cY'$ over $R(d')$ for some $d'>0$ such that $\cY'_k$ is a reduced divisor with strict normal crossings. Indeed, if we denote by $\loga{Y}$ the log space over $S(d)^{\dagger}$ obtained by endowing $\cY$ with the divisorial log structure induced by $\cY_k$, then we can first make the log structure on $\loga{Y}$ Zariski by means of a $\mu_d$-equivariant log blow-up as in \cite[5.6]{niziol}, and then apply a Galois-equivariant version of the semi-stable reduction theorem in \cite[1.8]{saito}. Since log blow-ups are \'etale and log differentials are compatible with fs base change, the result is still a Kulikov model of $X$. We do not include a detailed proof here because we will not need this property in the remainder of the paper.
\end{remark}

\sss \label{sss:kuliK3} If $X$ is a $K3$-surface, then there always exists a positive integer $d$ such that $X$ has a Kulikov model over $R(d)$ whose special fiber is a reduced strict normal crossings divisor \cite[2.1]{liedtke-matsumoto}. The special fibers of such models have been classified by Kulikov and Persson-Pinkham (see Section \ref{ss:semistable}), which explains our choice of terminology. However, we will see in Example \ref{ex:nokulikov} that $X$ may fail to have an equivariant Kulikov model, even when $X$ has potential good reduction.


\begin{thm}\label{thm:abkulikov}
Let $A$ be an abelian variety over $K$. Then $X$ has an equivariant Kulikov model $\cY$ over $R(d)$, for some $d>0$, such that $\cY$ is a regular proper scheme over $R(d)$ and $\cY_k$ is a reduced strict normal crossings divisor.
\end{thm}
\begin{proof}
We will use the notations from Section \ref{sec:abelian}.
 If $d$ is a sufficiently divisible positive integer, then $A(d)$ has semi-abelian reduction.  Replacing $d$ by a multiple if necessary,
 we can use Theorem 4.6 in \cite{kunnemann} to produce a model $\cY$ of $A(d)$ over $R(d)$ such that $\cY$ is a regular proper scheme over $R(d)$ and $\cY_k$ is a reduced strict normal crossings divisor. This model has the property that the $R(d)$-smooth locus of $\cY$ is a N\'eron model for $A(d)$, by
  the discussion in Section 4.4 in \cite{kunnemann}. Since the smooth locus of $\cY_k$ is dense, triviality of $\omega_{\cA(d)/R(d)}$ now implies that $\omega_{\cY/R}$ is trivial.
  Finally, we can arrange that the Galois action of $\Gal(K(d)/K)$ on $A(d)$ extends to $\cY$, by starting from a split ample degeneration $(\cA(d)^o,\mathcal{L},\mathcal{M})$ (in the sense of \cite[\S2.1]{kunnemann})  such that $\Gal(K(d)/K)$ acts on $\mathcal{L}$ and $\mathcal{M}$.
\end{proof}

\begin{remark}\label{rem:abkulikov}
One can show that, if $A$ is an abelian $K$-variety with semi-abelian reduction, then $A$ has a Kulikov model over $R$ in the category of schemes. Let $\cP$ be one of the proper regular $R$-models of $A$ constructed in \cite{kunnemann}. Then $\cP_k$ is a strict normal crossings divisor, and one can check that the line bundle $\omega_{\cP/R}(\cP_{k,\red}-\cP_k)$ is trivial, in the following way. Let $\omega$ be a volume form on $A$ that extends to a relative volume form on the N\'eron model. We claim that $\omega$ generates $\omega_{\cP/R}(\cP_{k,\red}-\cP_k)$. This can be checked after pulling back $\omega$ through the \'etale morphism of formal schemes $\widetilde{\cP}_{\mathrm{for}}\to \cP_{\mathrm{for}}$ (notations as in \cite[\S2.13]{kunnemann}).
 Thus it suffices to show that, if $E$ is the semi-abelian $K$-variety that uniformizes $A$ and $\widetilde{\omega}$ is a translation-invariant volume form on $E$ that extends to a relative volume form on the N\'eron lft-model of $E$, then
 $\widetilde{\omega}$ generates
  $\omega_{\widetilde{\cP}/R}(\widetilde{\cP}_{k,\red}-\widetilde{\cP}_k)$ for every relative completion $\widetilde{P}$ constructed as in \cite{kunnemann} by means of Mumford's method. This follows easily from the fact that toric schemes over $R(d)$ have trivial relative logarithmic canonical sheaf.
 It is also worth pointing out that Theorem \ref{thm:abkulikov} and Remark \ref{rem:abkulikov} remain valid if the residue characteristic of $R$ is positive (replacing $K(d)$ by a finite Galois extension of $K$).
\end{remark}

Kulikov models of abelian varieties are related to N\'eron models {\em via} the following result.

\begin{prop}
Let $A$ be an abelian variety over $K$ and let $\cX$ be a Kulikov model for $A$ over $R$. Then the $R$-smooth locus $\Sm(\cX)$ of $\cX$ is canonically isomorphic to the N\'eron model $\cA$ of $A$ (in particular, it is a scheme).
\end{prop}
\begin{proof}
By the universal property of N\'eron models, we know that, for every smooth $R$-scheme $\cY$, every morphism of $K$-schemes $\cY_K\to A$ extends uniquely to a morphism
of $R$-schemes $\cY\to \cA$. The uniqueness of the extension guarantees that this still holds when $\cY$ is a smooth algebraic space over $R$ (by gluing the morphisms we get on \'etale charts). In particular, the isomorphism $\cX_K\to A$ extends uniquely to a morphism of algebraic spaces $h:\Sm(\cX)\to \cA$. We will show that $h$ is an isomorphism.

 The map $\cX(R)\to \cX_K(K)$ is bijective because $\cX$ is proper over $R$. Since $\cX$ is regular, every $R$-point on $\cX$ factors through $\Sm(\cX)$, by \cite[3.1.2]{BLR}.
  On the other hand, we also have that $\cA(R)\to A(K)$ is bijective by the N\'eron mapping property, and the reduction map $\cA(R)\to \cA(k)$ is surjective by smoothness of $\cA$ and the fact that $R$ is henselian. Thus $h$ is surjective, and it is enough to show that it is an open immersion.

   Since $h$ is an isomorphism on the generic fibers, we only need to prove that it is \'etale: then $\Sm(\cX)$ is a scheme \cite[II.6.17]{knutson} so that fact (c) after the statement of \cite[3.5/7]{BLR} implies that $h$ is an open immersion. It suffices to show that the morphism of line bundles
   $$\alpha:h^*\omega_{\cA/R}\to \omega_{\Sm(\cX)/R}$$ is an isomorphism. The source and target of this morphism are trivial line bundles on $\Sm(\cX)$ and the generic fiber  $\cX_K$ is proper, so that it is enough to show that $\alpha$ is surjective on the stalks at some point $x$ of $\Sm(\cX)_k$. Since $k$ has characteristic zero, we can take for $x$ the generic point of any connected component of $\Sm(\cX)_k$ that is not contracted by $h$; such components exist because $h$ is surjective and $\Sm(\cX)$ and $\cA$ have the same dimension. This concludes the proof.
\end{proof}

\subsection{Toroidal models in the Nisnevich topology}
\sss Let $U$ be an algebraic space. A Nisnevich cover of $U$ is a family of \'etale morphisms  $\{f_{\alpha}:V_{\alpha}\to U\}_{\alpha\in A}$
 such that, for every point $u$ in $U$, there exist an element $\alpha$ in $A$ and a point $v$ in $V_{\alpha}$ such that $f_{\alpha}(v)=u$ and the extension of residue fields
 $\kappa(u)\to \kappa(v)$ is trivial. We will express this condition by saying that the point $u$ lifts to a point $v$ on $V_{\alpha}$. Nisnevich covers generate a Grothendieck topology on the category of algebraic spaces, called the Nisnevich topology.

 \sss Even though algebraic spaces are defined as sheaves on the \'etale site, they have very good properties already with respect to the Nisnevich topology: on every Noetherian algebraic space with a finite group action, we can find ``good'' equivariant Nisnevich charts that are affine schemes. To make this precise, we first need to introduce some terminology.
  Let $G$ be a finite group and let $f:V\to U$ be an equivariant morphism of Noetherian algebraic spaces with $G$-action. We say that $f$ is {\em fixed point reflecting} (fpr) if for every point $v$ of $V$, the stabilizer of $G$ at $v$ is equal to the stabilizer of $G$ at $f(v)$.

 \begin{prop}\label{prop:nisnevich}
 Let $U$ be a Noetherian algebraic space with an action of a finite group $G$.
  Then there exists an  fpr $G$-equivariant Nisnevich cover $V\to U$ such that $V$ is an affine scheme.
 \end{prop}
 \begin{proof}
 By \cite[2.1]{matsuura}, we can find an fpr $G$-equivariant \'etale surjection $V\to U$ such that $V$ is an affine scheme.
 Let $u$ be a point of $U$.
 An inspection of the proof of
 \cite[2.1]{matsuura} shows that we can choose $V\to U$ in such a way that $u$ lifts to a point on $V$. Then there exists a neighbourhood of $u$ in the Zariski closure of $\{u\}$ such that every point in this neighbourhood lifts to $V$. By Noetherian induction, we can now construct a morphism $V\to U$ as in the statement of the proposition.
 \end{proof}

 \if false
 \begin{prop}\label{prop:nis}
 Let $W\to V$ and $V\to U$ be morphisms of Noetherian algebraic spaces.
 \begin{enumerate}
\item \label{it:compos} If $W\to V$ and $V\to U$ are Nisnevich covers, then so its $W\to U$.
\item \label{it:local} If $W\to V$ and $W\to U$ are Nisnevich covers, then so is $V\to U$.
\item \label{it:basechange} If $U'\to U$ is a flat morphism of Noetherian algebraic spaces such that every point on $U$ lifts to a point on $U'$, then $V\to U$ is a Nisnevich cover if and only if $V\times_U U'\to U'$ is a Nisnevich cover.
\end{enumerate}
 \end{prop}
 \begin{proof}
Points \eqref{it:compos} and \eqref{it:local} are clear. It is also obvious that the property of being a Nisnevich cover is stable under base change. Thus we only need to prove
 that, in point \eqref{it:basechange}, the morphism $V\to U$ is a Nisnevich cover if $V\times_U U'\to U'$ is a Nisnevich cover. Every point of $U$ lifts to $V\times_U U'$, and this to $V'$. The morphism $V\to U$ is \'etale because this property can be checked after faithfully flat base change. Thus $V\to U$ is a Nisnevich cover.
 \end{proof}
 \fi

  The importance of fpr morphisms lies in the following property.
 \begin{prop}\label{prop:inert}
 Let $G$ be a finite group, acting on an algebraic space $U$ of finite type over a Noetherian scheme $Z$ with trivial $G$-action. Assume that the order of $G$ is invertible on $U$, and
 let $V\to U$ be an fpr $G$-equivariant Nisnevich cover. Then the quotient morphism $V/G\to U/G$ is a Nisnevich cover.
 \end{prop}
 \begin{proof}
  The quotients $U/G$ and $V/G$ are representable by algebraic spaces by \cite[2.2]{matsuura}. The construction in \cite[2.2]{matsuura} also implies that
  they are of finite type over $Z$, and hence Noetherian, because this holds in the case where $U$ and $V$ are schemes \cite[V.1.5]{sga1}.
 The morphism $V/G\to U/G$ is \'etale; this was stated in \cite[p.183]{knutson} and a proof can be found in \cite[2.17]{kollar}.
  Thus we only need to show that every point $u$ in $U/G$ lifts to a point in $V/G$.     Let $u'$ be a point in $U$ that maps to $u$.
  The point $u'$ lifts to a point $v'$ in $V$ because $V\to U$ is a Nisnevich cover. We denote by $v$ the image of $v'$ in $V/G$. Since $V\to U$ is also fpr, the stabilizer subgroups of $G$ at $u'$ and $v'$ coincide; we denote this stabilizer by $H$.

 We claim that $\kappa(u)=\kappa(u')^H=\kappa(v')^H=\kappa(v)$. This implies that $u$ lifts to the point $v$ on $V/G$. To prove our claim,
  we may assume that $U$ and $V$ are affine schemes,
  since the quotient maps $U\to U/G$ and $V\to V/G$ are affine and the formation of geometric quotients commutes with flat base change.
    Then by \cite[V.2.2(i)]{sga1}, we may also assume that $G=H$. Now our claim follows from the fact that, for every equivariant surjective morphism
    $A\to \kappa$ of rings with $G$-action such that the order of $G$ is invertible in $A$, the induced morphism $A^G\to \kappa^G$ is still surjective (this is the only place where we used the assumption that the order of $G$  is invertible on $U$).
  \end{proof}

\begin{definition}\label{def:toroidal}
Let $X$ be a smooth and proper $K$-scheme, and let $\cX$ be a normal proper $R$-model of $X$ in the category of algebraic spaces. We say that $\cX$ is {\em Nis-toroidal} if there exist
\begin{enumerate}
\item a Nisnevich cover of $\cX$ by finitely many schemes $\cU_1,\ldots,\cU_r$;
\item on each scheme $\cU_j$,
a reduced divisor $D_j$ such that $D_j$ is flat over $R$ and such that $\cU_j$, endowed with the divisorial {\em Zariski} log structure associated with $(\cU_j)_k+D_j$, is smooth over
$S^{\dagger}$.
\end{enumerate}
\end{definition}

\subsection{Proof of the Monodromy Property}

\begin{theorem}\label{thm:toroidal}
Let $X$ be a geometrically connected, smooth and proper $K$-scheme, and assume that $X$ has an equivariant Kulikov model over $R(d)$ for some $d>0$.
 Then $X$ has a Nis-toroidal proper $R$-model $\cX$ such that the divisor class $K_{\cX/R}+\cX_{k,\red}$ is torsion, where $K_{\cX/R}$ denotes the relative canonical divisor of $\cX$ over $R$.
\end{theorem}
\begin{proof}
 We will use two key concepts from Gabber's theory exposed in \cite{gabber}: {\em very tame} group actions on log schemes \cite[Exp.VI]{gabber} and the notion of {\em rigidification} \cite[Exp.VIII]{gabber}. Some care is required here, because the results in \cite{gabber} are formulated for log schemes and we want to work with
  log algebraic spaces. We will explain how the necessary results can be adapted to our set-up.

  Let $\loga{Z}$ be a fine and saturated log algebraic space of characteristic zero and let $G$ be a finite group acting on $\loga{Z}$. Following the definition for log schemes in \cite[VI.3.1]{gabber}, we say that the action of $G$ is very tame if, for every geometric point $\overline{z}$ of $\loga{Z}$, the stabilizer $G_{\overline{z}}$ of $\overline{z}$ acts trivially on the characteristic monoid $\mathcal{M}^{\sharp}_{\loga{Z},\overline{z}}$ and on the entire log stratum of $\loga{Z}$ that contains the image of $\overline{z}$. Assume that $\loga{Z}$ is regular and that the action of $G$ on $\loga{Z}$ is very tame.
  Deligne and, independently, Matsuura have proven that the quotient $\cZ/G$ exists in the category of algebraic spaces (see \cite{matsuura} and \cite[p.183]{knutson}). We denote by $\pi:\cZ\to \cZ/G$ the quotient map. Then we can form the quotient $\loga{Z}/G$ as in \cite[VI.3.2]{gabber}
   by endowing $\cZ/G$ with the log structure $$(\pi_*\mathcal{M}_{\loga{Z}})^G\to  (\pi_*\mathcal{O}_{\cZ})^G=\mathcal{O}_{\cZ/G}.$$
 This log structure is fine and saturated and the log algebraic space $\loga{Z}/G$  is still regular: the result in \cite[VI.3.2]{gabber} for schemes can be generalized to algebraic spaces by taking $G$-equivariant \'etale charts on $\cZ$ as in \cite[2.2]{matsuura} and \cite[p.183]{knutson}.

 Let $\cY$ be an equivariant Kulikov model for $X$ over $R(d)$, for some $d>0$, and denote by $\loga{Y}$ the log algebraic space we obtain by endowing $\cY$ with the divisorial log structure induced by $\cY_k$. We write $S(d)^{\dagger}$ for the scheme $S(d)=\Spec R(d)$ with its standard log structure. Then $\loga{Y}$ is smooth over $S(d)^{\dagger}$, the logarithmic relative canonical line bundle $\omega_{\loga{Y}/S(d)^{\dagger}}\cong \omega_{\cY/S(d)}(\cY_{k,\red}-\cY_k)$ is trivial, and $\mu_d$ acts on $\loga{Y}$.
  We say that the log structure on $\loga{Y}$ is {\em Nisnevich} if the sheaf of monoids $\mathcal{M}_{\loga{Y}}$
  that defines the log structure is the pullback of a sheaf on the Nisnevich site on $\cY$. We can make the log structure on $\loga{Y}$ Nisnevich by means of a canonical (and, in particular, $\mu_d$-equivariant) log blow-up: we first choose an fpr $\mu_d$-equivariant Nisnevich cover $\cU\to \cY$  such that $\cU$ is an affine scheme, using Proposition \ref{prop:nisnevich}. We pull back the log structure on $\loga{Y}$ to define a log scheme $\loga{U}$. We then use the canonical log blow-up in \cite[5.6]{niziol} to make the log structure on $\loga{U}$ Zariski. Log blow-ups are \'etale, and therefore do not affect the smoothness of $\loga{U}$ or the triviality of the logarithmic relative canonical line bundle. By a similar procedure, we can also make the log structure on $\loga{U}$ $\mu_d$-strict, which means that, for every irreducible component $E$ of $\cU_k$ and every $\zeta\in \mu_k$, either $\zeta E=E$ or $\zeta E\cap E=\emptyset$. This can always be achieved by performing a log modification corresponding to a barycentric subdivision of the fan of $\loga{U}$. Since each step in this procedure is canonical, the resulting morphism of log schemes $\loga{V}\to \loga{U}$ descends to a $\mu_d$-equivariant proper \'etale morphism of log algebraic spaces $\loga{Z}\to \loga{Y}$ that is an isomorphism on the generic fibers.
    The log algebraic space $\loga{Z}$ is smooth over $S(d)^{\dagger}$, its log structure is Nisnevich and $\mu_d$-strict, and it has trivial logarithmic relative canonical line bundle $\omega_{\loga{Z}/S(d)^{\dagger}}$.

  We claim that the quotient $\cZ/\mu_d$ is Nis-toroidal. If the action of $\mu_d$ on $\loga{V}$ is very tame, then $\loga{V}/\mu_d$ is regular, and hence smooth over $S^{\dagger}$ because $R$ has equal characteristic zero (see the proof of \cite[3.2.4]{BuNi}).
   Unfortunately, the action of $\mu_d$ on $\loga{V}$ need not be very tame, but we can make it very tame (at least Zariski-locally on $\cV$) by adding horizontal components to the logarithmic boundary; this process is called {\em rigidification} in   \cite[Exp.VIII]{gabber}.

   The proof of Lemma 5.3.8 in \cite[Exp.VIII]{gabber} produces an fpr $\mu_d$-equivariant surjective morphism of affine log schemes $h:\loga{W}\to \loga{V}$ such that the following properties hold:
    \begin{enumerate}
   \item $h$ is locally an open immersion on the underlying schemes;
   \item the log structure on $\loga{W}$ is the divisorial log structure induced by a strict normal crossings divisor on $\cW$ whose support contains the reduced special fiber $\cW_{k,\red}$ (thus $\loga{W}$ is regular and its log structure is Zariski);
   \item the $\mu_d$-action on $\loga{W}$ is very tame.
    \end{enumerate}
    Here we are using that the log structure on $\loga{V}$ is Zariski and that the group $\mu_d$ is split over $R$ to get a morphism $h$ that is locally an open immersion, rather than merely \'etale as in \cite[VIII.5.3.8]{gabber}.

    The very tameness of the $\mu_d$-action implies that $\loga{W}/\mu_d$ is regular and that the projection morphism $\loga{W}\to \loga{W}/\mu_d$ is \'etale \cite[VI.3.2]{gabber}. It also implies that $\loga{W}$ is $\mu_d$-strict; this follows, more precisely, from the fact that the stabilizer at each point $w$ acts trivially on the characteristic monoid $\mathcal{M}^{\sharp}_{\cW,w}$. Hence, the log structure on $\loga{W}/\mu_d$ is Zariski.
         Proposition \ref{prop:inert} implies that the morphism of algebraic spaces $\cW/\mu_d\to \cZ/\mu_d$ is a Nisnevich cover, because $\cW\to \cZ$ is an fpr Nisnevich cover. Thus $\cZ/\mu_d$ is Nis-toroidal; we set $\cX=\cZ/\mu_d$ and we define $\loga{X}$ by endowing $\cX$ with the log structure induced by $\cX_k$.

         It remains to prove that the divisor class $K_{\cX/R}+\cX_{k,\red}$ is torsion. We will prove that it is linearly equivalent to a rational multiple of the special fiber $\cX_k$. We define the rank one reflexive sheaf $\omega_{\loga{X}/S^{\dagger}}$ to be the pushforward to $\cX$ of the logarithmic relative canonical line bundle on the smooth locus of $\loga{X}\to S^{\dagger}$. Note that this smooth locus contains all the codimension one points of $\cX$, because $\cX$ and $\cX_{k,\red}$ are regular at these points.
         Let $\omega$ be a volume form on $X$. Then we must show that
         the divisor of $\omega$, viewed as a rational section of $\omega_{\loga{X}/S^{\dagger}}$, is a rational multiple of $\cX_{k}$.
        We denote by $\times^{\mathrm{fs}}$ the fiber product in the category of fine and saturated log schemes.  The morphism $$f:\loga{Z}\to \loga{X}\times^{\mathrm{fs}}_{S^{\dagger}} S(d)^{\dagger}$$ is an isomorphism over each generic point $\xi$ of $\cX_k$, because $f$ is finite and an isomorphism on the generic fibers, and $\loga{X}$ is log smooth over $S^{\dagger}$ at $\xi$ so that $\loga{X}\times^{\mathrm{fs}}_{S^{\dagger}} S(d)^{\dagger}$ is normal at each point lying over $\xi$. Since the logarithmic relative canonical line bundle is compatible with fine and saturated base change, it suffices to show that the divisor of $\omega\otimes_K K(d)$ on $\loga{Z}$ is a rational multiple of the special fiber of $\cZ$. This follows at once from the fact that $\omega_{\loga{Z}/S(d)^{\dagger}}$ is trivial.
\end{proof}

\begin{theorem}\label{thm:main}
Let $X$ be a geometrically connected, smooth and proper $K$-scheme with trivial canonical line bundle, and assume that $X$ has an equivariant Kulikov model over $R(d)$ for some $d>0$.
 Let $\omega$ be a volume form on $X$. Then $Z_{X,\omega}(T)$ lies in the ring
 $$\gro\left[T,\frac{1}{1-\LL^aT^b}\right]_{(a,b)\in \Z\times \Z_{>0},\,a/b=1-\min(\omega)}.$$
  Thus $1-\min(\omega)$ is the only pole of $Z_{X,\omega}(T)$.
 \end{theorem}
 \begin{proof}
  We will deduce this result from Theorem \ref{thm:toroidal} and the computation of the motivic zeta function on a log smooth model in \cite{BuNi}.
 Let $\cX$ be a strictly toroidal proper $R$-model of $X$ such that the divisor class $K_{\cX/R}+\cX_{k,\red}$ is torsion. Then
         the divisor $\mathrm{div}_{\cX}(\omega)$ of $\omega$, viewed as a rational section of the rank one reflexive sheaf $\omega_{\loga{X}/S^{\dagger}}$, is equal to $\alpha\cX_{k}$ for some rational number $\alpha$. We will prove Theorem \ref{thm:main} with $1-\min(\omega)$ replaced by $-\alpha$;  it then follows automatically from Theorem \ref{thm:largest} that $\alpha=\min(\omega)-1$.

 Choose $\cU_j$ and $D_j$ as in Definition \ref{def:toroidal}, for  $j\in \{1,\ldots,r\}$. After a suitable refinement of the cover $\{\cU_1,\ldots,\cU_r\}$, we can find a partition of $\cX_k$ into subschemes $V_1,\ldots,V_s$ and, for every $\ell$ in $\{1,\ldots,s\}$, an element $j(\ell)$ in $\{1,\ldots,r\}$ such that $\cU_{j(\ell)}\times_{\cX}V_\ell\to V_\ell$ is an isomorphism.
   Then $Z_{X,\omega}(T)$ satisfies the following additivity property with respect to the Nisnevich cover $\{\cU_1,\ldots,\cU_r\}$ of $\cX$.
 For each $j$, we consider the motivic zeta function $Z^{\widehat{\mu}}_{\cU_j,\omega}(T)$ as defined in \cite[\S6.2]{BuNi}; here we abuse notation by writing $\omega$ for the restriction of $\omega$ to the generic fiber of $\cU_j$. This zeta function is a formal power series in $T$ with coefficients in the localized Grothendieck ring $\mathcal{M}^{\widehat{\mu}}_{(\cU_j)_k}$ of varieties over $(\cU_j)_k$ with good $\widehat{\mu}$-action.  For every $\ell$ in $\{1,\ldots,s\}$, we define the generating series $Z_{\ell}(T)$ by first applying the base change morphism $$\mathcal{M}^{\widehat{\mu}}_{(\cU_{j(\ell)})_k}\to \mathcal{M}^{\widehat{\mu}}_{V_\ell}$$ to the coefficients of $Z^{\widehat{\mu}}_{\cU_j,\omega}(T)$, and then the forgetful morphism $$\mathcal{M}^{\widehat{\mu}}_{V_\ell}\to \mathcal{M}^{\widehat{\mu}}_{k}.$$ Then it follows from Proposition \ref{prop:add} in the appendix  that
 $$Z_{X,\omega}(T)=Z_{1}(T)+\ldots +Z_{s}(T).$$ Thus it suffices to show that each of the zeta functions
 $Z_{\cU_j,\omega}(T)$ lies in $$\mathcal{M}^{\widehat{\mu}}_{(\cU_j)_k}\left[T,\frac{1}{1-\LL^aT^b}\right]_{(a,b)\in \Z\times \Z_{>0},\,a/b=-\alpha}.$$
 To simplify the notation, we will fix an index $j$ in $\{1,\ldots,r\}$ and write $\cU$ and $D$ instead of $\cU_j$ and $D_j$.

 Denote by
 $\loga{U}$ the space $\cU$ endowed with the divisorial Zariski log structure induced by $\cU_k+D$. By the definition of a strictly toroidal model, $\loga{U}$ is
smooth over $S^{\dagger}$.
 Write $\cU_k=\sum_{i\in I}N_i E_i$. For every $i\in I$, we denote by $\nu_i$ the order of $\omega$ along  $E_i$, where we view $\omega$ as a rational section of the line bundle
 $\omega_{\loga{U}/S^{\dagger}}$. Then $\sum_{i\in I}\nu_i E_i$ is the pullback of $\mathrm{div}_{\cX}(\omega)$ to $\cU$, so that $\nu_i=\alpha N_i$ for every $i\in I$.
 Now \cite[6.2.2]{BuNi} implies that $Z_{\cU,\omega}(T)$ lies in the ring
 $$\mathcal{M}^{\widehat{\mu}}_{\cU_k} \left[T,\frac{1}{1-\LL^aT^b}\right]_{(a,b)\in \Z\times \Z_{>0},\,a/b=-\alpha}.$$
 \end{proof}
\begin{cor}\label{cor:MP}
Let $X$ be a geometrically connected, smooth and proper $K$-scheme with trivial canonical line bundle, and assume that $X$ has an equivariant Kulikov model over $R(d)$ for some $d>0$. Then $X$ satisfies the Monodromy Property in Definition \ref{def:MP}.
\end{cor}
\begin{proof}
This is an immediate consequence of Theorems \ref{thm:hodge} and \ref{thm:main}.
\end{proof}

\sss Corollary \ref{cor:MP} yields, in particular, a new proof of the Monodromy Property for abelian varieties, by the existence of equivariant Kulikov models for abelian varieties (Theorem \ref{thm:abkulikov}). However, the proof we have given in Section \ref{sec:abelian} provides finer information on the coefficients of the zeta function, which is why we have included it in the paper.

\begin{exam}\label{ex:nokulikov}
There are examples of geometrically connected, smooth and proper $K$-schemes with trivial canonical line bundle that do not have an equivariant Kulikov model over $R(d)$ for any $d>0$, and in such examples, the motivic zeta function may have more than one pole.
 For instance, consider the closed subscheme $\cY$ of $\mathrm{Proj}\,R[x,y,z,w]$ defined by the homogeneous equation
$$ x^2 w^2 + y^2 w^2 + z^2 w^2 + x^4 + y^4 + z^4 + t w^4 =0.$$
 One checks by direct computation that $ \mathscr{Y} $ is regular.
   Its generic fiber $X=\cY_K$ is a smooth $K3$ surface, and its special fiber $ \mathscr{Y}_k $ a singular $K3$ surface with a unique singularity at $O=(0:0:0:1) $, of type $A_1$. We can construct an snc-model $\cX$ of $X$ by blowing up $\cY$ at $O$. The strict transform $D$ of $\mathscr{Y}_k$ is smooth and intersects the exceptional divisor $E\cong \mathbb{P}^2_k$ transversally along a smooth conic $C$, and we have $\cX_k=D+2E$.

 Let $\omega$ be a volume form on $X$ that extends to a relative volume form on $\cY$ at the generic point of $\cY_k$. Then the motivic zeta function $Z_{X,\omega}(T)$ has two poles, namely, $0$ and $-1/2$. Indeed, applying the formula in Theorem \ref{thm-snc} to the snc-model $\cX$, we find
 $$Z_{X,\omega}(T)=[D^o]\frac{T}{1-T}+[\widetilde{E}^o]\frac{\LL^{-1}T^2}{1-\LL^{-1}T^2}+[C]\frac{\LL^{-1}T^3}{(1-T)(1-\LL^{-1}T^2)}.$$

 In particular, it follows from Theorem \ref{thm:main} that $X$ does not have an equivariant Kulikov model over $R(d)$ for any $d>0$ (although it has good reduction over $R(2)$ in the category of algebraic spaces). However, the surface $X$ still satisfies the Monodromy Property: computing the monodromy zeta function of $X$ on the snc-model $\cX$ using the A'Campo formula from \eqref{sss:acampo}, we see that $-1$ is an eigenvalue of the monodromy transformation $\sigma$ on $H^2(X\times_K K^a,\Q)$.
\end{exam}

\section{Examples}\label{sec:examples}
\subsection{The semi-stable case}\label{ss:semistable}
\sss Let $X$ be a geometrically connected smooth proper $K$-scheme with trivial canonical line bundle, and let $\omega$ be a volume form on $X$. Assume that $X$ has semi-stable reduction.
  In this case, the Monodromy Property in Definition \ref{def:MP} is trivially satisfied: from the formula for the motivic zeta function in Theorem \ref{thm-snc}, we immediately see that all the poles of $Z_{X,\omega}(T)$ are integers. Nevertheless, it is still
interesting to investigate how the geometry of $X$ is reflected in other aspects of the motivic zeta function.

\sss The case of $K3$-surfaces was studied in detail by Stewart and Vologodsky in \cite{StVo}. Their analysis is based on the so-called {\em Kulikov classification} of semi-stable degenerations of $K3$ surfaces. We have already recalled in \eqref{sss:kuliK3} that every $K3$-surface $X$ has a Kulikov model over a finite extension $K'$ of $K$. We do not know if it suffices to take an extension $K'$ where $X$ acquires semi-stable reduction: the algebro-geometric construction of Kulikov models consists of running a Minimal Model Program on a semi-stable model and resolving the singularities by means of small resolutions, and the latter step may require an additional extension of $K'$ (see the proof of \cite[2.1]{liedtke-matsumoto}). In any case, if $X$ has a Kulikov model $\cX$ over $R$, then the possible special fibers $\cX_k$ have been completely classified by Kulikov, Persson and Pinkham. They are subdivided into three types:
\begin{itemize}
\item {\em Type I}. The special fiber $\cX_k$ is smooth. This happens if and only if $\sigma$ acts trivially on $H^2(X\times_K K^a,\Q)$.

\item {\em Type II}. The special fiber $\cX_k$ is a chain of surfaces; the interior surfaces are elliptic ruled, the outer ones are rational, and the intersection curves are elliptic. This happens if and only if $\sigma$ has a Jordan block of rank $2$ on $H^2(X\times_K K^a,\Q)$, but no Jordan block of rank $3$.

 \item {\em Type III}. The special fiber $\cX_k$ is a union of rational surfaces, and its dual intersection complex is a triangulated $2$-sphere. This happens if and only if $\sigma$ has a Jordan block of rank $3$ on $H^2(X\times_K K^a,\Q)$.
\end{itemize}
 Using \eqref{eq:rescale}, we can rescale $\omega$ in such a way that $\min(\omega)=1$.
 Then, in the Type I case, the motivic zeta function $Z_{X,\omega}(T)$ equals $[\cX_k]T(1-T)^{-1}$. In the Type II and III cases, Stewart and Vologodsky gave an elegant formula for the motivic zeta function $Z_{X,\omega}(T)$ in terms of the limit mixed Hodge structure associated with $X$: the coefficients of $Z_{X,\omega}(T)$ are computed in Theorem 1 of \cite{StVo}.

 The Kulikov classification provides in particular a description of the essential skeleton $\Sk(X)$. This description can also be proven directly, by means of the following result (which does not assume that $X$ has a Kulikov model over $R$).
\begin{thm}\label{thm:sktopology}
Let $X$ be a geometrically connected, smooth and projective $K$-variety of dimension $n$ with trivial canonical line bundle.
\begin{enumerate}
\item \label{it:psman1} The essential skeleton $\Sk(X)$ is a connected pseudo-manifold with boundary.
\item \label{it:psman3}  The degeneracy index $\delta(X)$ equals $n$ if and only if there is a Jordan block of monodromy on $H^n(X\times_K K^a,\Q)$ of size $n+1$.
\item \label{it:psman4} If $X$ has semi-stable reduction and $\delta(X)=n$, then $\Sk(X)$ is a closed pseudo-manifold. If, moreover, $h^{i,0}(X)=0$ for $0<i<n$, then $\Sk(X)$ has the rational homology of the $n$-dimensional sphere $S^n$.
\item \label{it:psman2} Assume that $X$ has semi-stable reduction and
 trivial geometric fundamental group (that is, $\pi_1^{\mathrm{\acute{e}t}}(X\times_K K^a)=\{1\}$). Then the profinite completion of $\pi_1(\Sk(X))$ is trivial.
 \end{enumerate}
\end{thm}
\begin{proof}
\eqref{it:psman1} This is proven in \cite[4.1.4]{NiXu}.

\eqref{it:psman3} It suffices to prove this statement after base change to some finite extension of $K$, by Proposition \ref{prop:bc}. Thus, we may assume that $X$ has a projective semi-stable model, by the semi-stable reduction theorem. In that case, the result was proven in \cite[4.1.10]{NiXu}.

\eqref{it:psman4} This follows from \cite[4.1.7 and 4.1.10]{NiXu} if $X$ has a projective semi-stable model, and we can reduce to this case by means of Proposition \ref{prop:bc} and the semi-stable reduction theorem.

\eqref{it:psman2}
    We denote by $\widehat{K^a}$ the completion of $K^a$. Then $\pi_1^{\mathrm{\acute{e}t}}(X\times_K \widehat{K^a})=\{1\}$ by the invariance of the fundamental group under change of algebraically closed base field for proper schemes (or for separated schemes of finite type in characteristic zero).
    The essential skeleton $\Sk(X)$ is homotopy equivalent to $X^{\an}$, and the morphism $(X\times_K \widehat{K^a})^{\an}\to X^{\an}$ is a homotopy equivalence because $X$ has semi-stable reduction (this follows, for instance, from Proposition \ref{prop:bc}). Every non-trivial finite topological cover of  $(X\times_K \widehat{K^a})^{\an}$ gives rise to a non-trivial finite \'etale cover of $X\times_K \widehat{K^a}$ by non-archimedean GAGA \cite[3.4.13]{berkbook}. It follows that the profinite completion of $\pi_1(\Sk(X))$ is trivial.
\end{proof}

\begin{cor}\label{cor:skeletonK3}
Let $X$ be a $K3$-surface with semi-stable reduction. Then the maximal size of a Jordan block of monodromy on $H^2(X\times_K K^a,\Q)$ is equal to $\delta(X)+1$. If $\delta(X)=0$ then $\Sk(X)$ is a point; if $\delta(X)=1$ then $\Sk(X)$ is homeomorphic to the interval $[0,1]$; if $\delta(X)=2$ then $\Sk(X)$ is homeomorphic to $S^2$, the two-dimensional sphere.
\end{cor}
\begin{proof}
 Theorem \ref{thm:sktopology} implies that $\delta(X)=2$ if and only if there is a Jordan block of monodromy on $H^2(X\times_K K^a,\Q)$ of size $3$. It also states that, in that case,  $\Sk(X)$ is a $2$-dimensional closed pseudo-manifold and a rational homology $2$-sphere. This implies that $\Sk(X)$ is homeomorphic to $S^2$.

To finish the proof, it suffices to show that $\delta(X)=0$ if and only if the monodromy action on $H^2(X\times_K K^a,\Q)$ is trivial; the remainder of the statement then follows from Theorem \ref{thm:sktopology}. If the monodromy action on $H^2(X\times_K K^a,\Q)$ is trivial, then Theorem \ref{thm:hodge} implies that $\delta(X)=0$. Assume, conversely, that $\delta(X)=0$. We must show that the monodromy action on $H^2(X\times_K K^a,\Q)$ is trivial. We already know that it is unipotent, by the assumption that $X$ has semi-stable reduction. Thus it is enough to prove the triviality of the monodromy after a finite extension of $K$, so that we can assume that $X$ has a projective semi-stable model $\cX$ over $R$. The unique point of $\Sk(X)$ corresponds to an irreducible component $E$ of $\cX_k$. Let $\omega$ be a generator for the module of logarithmic relative canonical forms
on $\cX$ over $R$. Then, by the definition of $\Sk(X)$, we know that $\omega$ restricts to a non-zero canonical form on $E$ and vanishes along all the other components of $\cX_k$.
 Now it follows from the Clemens-Schmid exact sequence that the monodromy action on $H^2(X\times_K K^a,\Q)$ is trivial (see \cite[2.7.5]{persson} -- one can reduce to the case where $\cX$ is defined over an algebraic curve in the same way as in the proof of Theorem \ref{thm:hodge}).
  \end{proof}

\sss Corollary \ref{cor:skeletonK3} has a partial generalization to higher dimensions. Let $X$ be a geometrically connected, smooth and projective $K$-variety of dimension $n$ with trivial canonical bundle and trivial geometric fundamental group. Assume, that $X$ has semi-stable reduction, $h^{i,0}(X)=0$ for $0<i<n$ and  $\delta(X)=\dim(X)$ (the last condition is sometimes expressed by saying that $X$ is {\em maximally degenerate} or {\em maximally unipotent}). Then Kontsevich and Soibelman's non-archimedean interpretation of the SYZ conjecture in the theory of mirror symmetry suggests that $\Sk(X)$ is homeomorphic to the $n$-sphere $S^n$.  If $n=3$, this was proven by Koll\'ar and Xu in  \cite[\S34]{kollar-xu}. Note that, in that case, it suffices to prove that $\Sk(X)$ is a topological manifold, by residual finiteness of the fundamental groups of $3$-manifolds and the Poincar\'e conjecture. Koll\'ar and Xu also proved the $n=4$ case under the  assumption that $X$ has a minimal $dlt$-model that is $snc$.

\subsection{Triple-point free degenerations of $K3$-surfaces}
\sss Example \ref{ex:nokulikov} is a special case of a {\em triple-point free} degeneration of $K3$-surfaces. These are $K3$-surfaces over $K$ that have an snc-model $\cX$ over $R$ such that no three distinct irreducible components of $\cX_k$ intersect; equivalently, the dual intersection complex of $\cX_k$ has dimension at most $1$. Apart from the case of elliptic curves, this seems to be the only class of varieties with trivial canonical line bundle were a classification of snc-models has been made without the assumption of semi-stable reduction: this is the Crauder-Morrison classification in \cite{crauder-morrison}. More precisely, Crauder and Morrison classified triple-point free degenerations of surfaces with {\em numerically} trivial canonical bundle. Unfortunately, it is not clear from their work which combinatorial types of special fibers can really occur for degenerations of $K3$-surfaces.

\sss Minimal triple-point free degenerations $\cX$ of $K3$ surfaces $X$ (in the category of algebraic spaces) with $\delta(X)=0$ are called {\em flowerpot degenerations} because of the particular shape of the dual graph of the special fiber: a {\em pot} (the unique point in the essential skeleton $\Sk(X)$) and, attached to it, a finite number of {\em flowers} (corresponding to chains of surfaces). In the case $\delta(X)=1$, one speaks of {\em chain degenerations}: the essential skeleton $\Sk(X)$ is homeomorphic to a line segment, and in the dual graph of $\cX_k$ we again find a finite number of flowers emanating from $\Sk(X)$. We have verified the Monodromy Property for all the flowerpot degenerations and most of the chain degenerations, and we are currently investigating the remaining cases in collaboration with A.~Jaspers. The results will appear in Jaspers's PhD thesis; see \cite{annelies} for an announcement.

\subsection{Kummer surfaces and Hilbert schemes}

\sss In the remaining paragraphs of this section, we shall discuss several examples of Calabi-Yau varieties that admit equivariant Kulikov models after a suitable extension of $K$.
 We will first consider Kummer K3 surfaces; precisely, we shall prove that the Kummer surface associated to an abelian surface with potential good reduction admits a smooth equivariant Kulikov model after a suitable extension in the base. In the discussion below, we will follow closely \cite[Sec.~4]{Mat}.
 For any field $F$, we denote by $F^a$ a fixed algebraic closure of $F$. For any abelian variety $B$ over $F$, we denote by $\iota_B$ the multiplication by $-1$ on $B$.


\begin{definition}
Let $X$ be a surface defined over a field $F$ of characteristic different from $2$. We say that $X$ is a Kummer surface if there exists an abelian surface $ B$ over $F^a $ such that $ X \times_F F^a $ is isomorphic to the minimal desingularization $ \mathrm{Km}(B) $ of the quotient surface $ B/H $, where $H$ is the group generated by $\iota_B$.
\end{definition}

\begin{remark} Alternatively, one can also construct $ \mathrm{Km}(B) $ in the following way. Let $ \tilde{B} $ denote the blow-up of $B$ in the subscheme $B[2]$ of $2$-torsion points. Then the action of $ H $ lifts uniquely to an action on $ \tilde{B} $ and one has that $ \mathrm{Km}(B) = \tilde{B}/H $.
\end{remark}

\sss\label{subsec:Kummer} Let $A$ be an abelian surface over $K$, and let $ X = \mathrm{Km}(A) $ be the Kummer surface associated to $A$. We assume that $A$ has \emph{potential good reduction}; by this we mean that there exists an integer $ d > 1 $ such that the N\'eron model $\cA(d) $ of $A \times_K K(d)$ is smooth and proper over $R(d)$. Then, by \cite[Lem.~4.2]{Mat}, also $X \times_K K(d) $ admits a smooth and proper $R(d)$-model. For later use, we provide a sketch of the argument.
 Denote by $ \iota $ the extension of $ \iota_{A \times_K K(d)} $ to $\cA (d)$, and by $H$ the group of $R(d)$-automorphisms of $\cA (d)$ generated by $ \iota $. Let $ \widetilde{\cA (d)}  \to \cA (d)$ be the blow-up in the closed subscheme $ (\cA (d))[2] $. Then the $H$-action lifts uniquely to $\widetilde{\cA (d)} $. By \cite[Lem.~1.2]{Mat}, $ \widetilde{\cA (d)} $ is smooth over $R(d)$, and the quotient $ \cX (d) := \widetilde{\cA (d)}/H $ forms a smooth and proper model of $ X \times_K K(d) $. Note, in particular, that the special fiber of $ \cX (d) $ is isomorphic to the Kummer surface associated to $ \cA (d)_k $.  In order to apply the results from Section \ref{sec:kulikov}, we need to check that the construction in (\ref{subsec:Kummer}) is Galois-equivariant.




\begin{prop}\label{prop:Kummer}
Let $A$ be an abelian surface over $K$ with potential good reduction. Then $ X = \mathrm{Km}(A)$ admits a smooth equivariant Kulikov model.
\end{prop}
\begin{proof}
We keep the notation in (\ref{subsec:Kummer}). It suffices to check that the smooth model $ \cX (d) $ is Galois equivariant. To do this, let us fix a generator $\sigma$ of $\mu_d$. Since $\sigma$ acts as a homomorphism on $\cA (d)$, it is straightforward to check that it commutes with the multiplication by $2$ (and with the involution $\iota$), so that $ (\cA (d))[2] $ is invariant under the action of $\sigma$. From this it follows that $ \sigma $ extends uniquely to an automorphism $ \tilde{\sigma} $ of $ \widetilde{\cA (d)} $ under which the exceptional locus is invariant, and commuting with the $H$-action on $ \widetilde{\cA (d)} $. Thus, $ \tilde{\sigma} $ descends to an automorphism of $ \cX (d) $. Since the formation of blow-up, resp.~quotient by $H$, commutes with base change from $K$ to $K(d)$, it is clear that this $\mu_d$-action on $ \cX (d) $ restricts to the obvious action on $ X \times_K K(d) $.
\end{proof}


Combining Proposition \ref{prop:Kummer} and Corollary \ref{cor:MP} yields the following result.

\begin{cor}
Let $A$ be an abelian surface over $K$ with potential good reduction, and let $ X = \mathrm{Km}(A) $. Then $X$ satisfies the Monodromy Property.
\end{cor}


\sss We would like to point out that even if $A$ does not have good reduction over $R$, it could still happen that $X$ has good reduction; this occurs precisely when $\mathrm{Gal}(K^a/K)$ acts on $H^1(A \times_K K^a, \mathbb{Q}_{\ell})$ by multiplication with $-1$. However, in this case, there exists a quadratic twist $B$ of $A$, with good reduction over $R$, and such that $ X = \mathrm{Km}(B) $ (cf.~e.g.~the proof of \cite[Thm.~4.1]{Mat}).

This means that, in order to get non-trivial examples, it suffices to assume that $ [K':K] > 2 $, where $K'$ denotes the minimal extension over which $A$ acquires good reduction. It is easy to find such examples. For instance, if $A$ is the Jacobian of a smooth, projective and geometrically connected curve $C/K$ of genus $2$, it suffices to assume that the stabilization index $e(C)$ is strictly greater than $2$ (see \cite{HaNi-book} for the definition of, and properties of, the invariant $e(C)$).

\sss We can formulate a similar result for Hilbert schemes of $n$ points on K3 and abelian surfaces. Recall that, for any integer $n \geq 1$, this construction yields a $2n$-dimensional smooth and proper variety with trivial canonical sheaf. In the K3 case, these varieties, and their deformations, form one of the main series of known examples of Irreducible Holomorphic Symplectic Varieties (IHSV). A similar statement is true when the underlying surface is abelian, after replacing the Hilbert scheme with its associated generalized Kummer variety.

\begin{prop}
Let $X$ be either a K3 surface or an abelian surface over $K$, and let $X^{[n]}$ denote the Hilbert scheme of $n$ points on $X$. We assume that:
\begin{enumerate}
\item the monodromy action on the cohomology of $X$ is non-trivial.

\item $X$ admits a smooth equivariant Kulikov model over $R(d)$, for some $d > 1$.
\end{enumerate}
Then properties $(1)$ and $(2)$ hold also for $X^{[n]}$.
\end{prop}
\begin{proof}
By our assumptions, there exists a smooth equivariant Kulikov model $\cY$ of $X \times_K K(d)$ over $R(d)$. Then the relative Hilbert scheme $\cY^{[n]}$ is a smooth and proper $R(d)$-model of $X^{[n]} \times_K K(d)$. By standard functorial properties of the Hilbert functor, it is also Galois-equivariant.
  The $\mu_d$-representation $$H^2(\cY_k^{[n]}, \mathbb{Q})\cong H^2(X^{[n]}\times_K K^a,\Q)$$ contains $ H^2(\cY_k, \mathbb{Q}) $ as a direct factor \cite[Prop.~6]{Beauville}. Thus if the monodromy action on the cohomology of $X$ is non-trivial, the same is true for $X^{[n]}$.
\end{proof}

\begin{cor}
Let $X$ be an abelian surface over $K$ with potential good reduction or a K3 surface with a smooth equivariant Kulikov model over $R(d)$ for some $d>0$.
 Then $X^{[n]}$ satisfies the Monodromy Property for every $n>0$.
\end{cor}

\subsection{Equivariant deformations}\label{subsec-equidef}
\sss \label{sss:equiv} We will next discuss another class of Calabi-Yau variaties admitting smooth equivariant Kulikov models, which, loosely said, arise by reversing the procedure of semi-stable reduction.
 More precisely, let $\cY_0$ be a connected smooth projective $k$-variety, equipped with an automorphism $\sigma$ of
 finite order $d\geq 2$.  We identify $\sigma$ with the canonical generator $\exp(2\pi i/d)$ of $\mu_d$.
 Let $\cY$ be a smooth and projective $R(d)$-scheme equipped with a lift of the $\mu_d$-action on $\Spec R(d)$ and with a $\mu_d$-equivariant isomorphism
 between $\cY_k$ and $\cY_0$. Then the generic fiber $\cY_{K(d)}$ descends to a geometrically connected smooth projective $K$-scheme $X$ over $K$.
  If $\cY_{0}$ has trivial canonical line bundle, then the same holds for $X$. By construction, $X$ has an equivariant Kulikov model over $R(d)$.
  If $\sigma$ acts non-trivially on the cohomology of $\cY_0$, then the monodromy action on the cohomology of $X$ is non-trivial, so that $X$ has no smooth Kulikov model over $R$.
  The condition that $\sigma$ acts non-trivially on the cohomology of $\cY_0$ is automatically satisfied when $\cY_0$ is a K3 surface or an abelian variety, because in those cases the automorphism group of $\cY_0$ acts faithfully on the cohomology ring of $\cY_0$.

\sss \label{subsec-moduli} Given $\cY_0$, it is straightforward to construct schemes $\cY$ satisfying the above properties: we can simply take
 $\cY=\cY_0\times_{k}R(d)$ endowed with the
 diagonal $\mu_d(k)$-action.  We can also construct more interesting (non-isotrivial) examples using moduli theory.
 Consider a Deligne-Mumford stack $ \mathcal{M} $ which is smooth, separated and of finite type over $k$. Let $x \in \mathcal{M}(k)$ be a point with (finite) stabilizer group $ G_x $. Then we can find a smooth affine $k$-scheme $ U = \mathrm{Spec}~A $ and an \'etale morphism
$$ \phi \colon ([U/G_x],w) \to (\mathcal{M},x) $$
such that $\phi$ induces an isomorphism of stabilizer groups at $w$ (cf.~e.g.~\cite[Thm.~1.1]{AHR}).

\begin{lemma}\label{lemma-equidef}
Assume that $G_x$ is cyclic. Then there exists a stacky curve $[C/G_x]$ in $[U/G_x]$ passing through $w$ and smooth in a neighbourhood of $w$.
\end{lemma}
\begin{proof}
Let $ \mathfrak{m} \subset A $ be the maximal ideal corresponding to $w$. By our assumptions on $G_x$, we can find an isomorphism of $G_x$-representations
$$ \phi \colon \oplus_{i=1}^n L_i \to \mathfrak{m}/\mathfrak{m}^2, $$
where each $L_i$ is an irreducible $1$-dimensional $G_x$-representation. After re-indexing, we can assume that $L_n$ is non-trivial as a representation. We can find an invariant finite dimensional subspace $ W $ in $ \mathfrak{m} $ surjecting onto $ \mathfrak{m}/\mathfrak{m}^2 $. Let $W'$ be the kernel of this map. Since $G_x$ is linearly reductive, there exists an invariant complement $V$ to $W'$ such that the induced $ V \to \mathfrak{m}/\mathfrak{m}^2 $ is an equivariant isomorphism. Hence, we can lift $\phi$ to a homomorphism
$$ \psi \colon \oplus_{i=1}^n L_i \to \mathfrak{m}. $$
The image $\psi(\oplus_{i=1}^{n-1} L_i) $ is an invariant ideal $I$ contained in $\mathfrak{m}$, and we can take $C = Z(I)$.
\end{proof}

\sss We remark that if the order of $G_x$ is prime, one can even assume that $L_n$ is a faithful representation. Also, if $G_x$ is not cyclic, we can instead take any cyclic subgroup $ H \subset G_x $, and obtain a morphism $ [U/H] \to [U/G_x] $. Applying Lemma \ref{lemma-equidef} again, we get
$$ [C/H] \to \mathcal{M}. $$
 Recall (see for instance the discussion in \cite[4.3]{Alper}) that if $ \mathcal{M} $ is representing a reasonable moduli functor, then, a morphism $ [C/H] \to \mathcal{M} $ is equivalent to a family $ f \colon \mathscr{Y} \to C $ in $\mathcal{M}(C)$, where $H$ acts on $\mathscr{Y}$ and $f$ is equivariant. In other words, by base change to the localization and completion at the fixed point $w \in C $, we get the model we are after.


\sss We will now apply the above results to K3 surfaces. For any integer $n > 0$, we denote by $\mathcal{M}_{2n}$ the moduli space parametrizing pairs $(Y,L)$ where $Y$ is a K3 surface and $L$ is a polarization of degree $2n$. By \cite{Riz}, $\mathcal{M}_{2n}$ is a Deligne-Mumford stack which is separated and of finite type over $\mathrm{Spec}~\mathbb{Z} $. It is, moreover, smooth over $\mathrm{Spec}~\mathbb{Z}[\frac{1}{2n}] $.

\begin{prop}\label{prop:equiv-def}
Let $\cY_0$ be a polarized K3 surface over $k$, and let $ G \cong \mu_p $ be a cyclic subgroup of $\mathrm{Aut}(\cY_0) $ fixing the polarization, with $p$ a prime. Then we can find a K3 surface $X$ over $ K = k\llpar t\rrpar $ such that the following hold:
\begin{enumerate}
\item $X$ admits a smooth equivariant Kulikov model $ \cY $ over $R(p)$.

\item There exists a $\mu_p$-equivariant isomorphism $ \cY \times_{R(p)} k \cong \cY_0$.

\item $X$ does not admit a smooth Kulikov model over $R$.

\item $X$ is not isotrivial; that is, $X\times_K K^a$ is not defined over $k$.
\end{enumerate}
\end{prop}
\begin{proof}
By Lemma \ref{lemma-equidef}, we can find a $\mu_p$-equivariant polarized deformation $ f \colon \cY \to \mathrm{Spec} R(p) $ of $\cY_0$ such that
the associated morphism to the moduli stack of polarized K3 surfaces is not constant.  The generic fiber $ Y = \cY \times_{R(p)} K(p)$ descends to a polarized K3 surface $X$ over $K$ that admits a smooth equivariant Kulikov model over $R(p)$, but not over $R$ (see the general discussion in \eqref{sss:equiv}).
 By construction, the base change to $K^a$ of the polarized K3 surface $X$ is not defined over $k$. But then the surface $X\times_K K^a$ itself is not defined over $k$, either:
 the morphism $\mathrm{Pic}(S)\to \mathrm{Pic}(S\times_k K^a)$ is an isomorphism for every $K3$ surface $S$ over $k$, since $\mathrm{Pic}(S)$ is a constant group scheme over $k$.
\end{proof}



\begin{remark}
It is reasonable to expect that similar results  hold for polarized IHSV-s of K3 type. Indeed, for any such variety $V$ over a field $F$ of characteristic $0$, the representation of $\mathrm{Aut}(V)$ on $H^2(V \times_F F^a,\mathbb{Z}_{\ell})$ is faithful (cf.~e.g.~\cite[Lem.~5]{HaTs}). Moreover, coarse moduli spaces can be be constructed, and share most of the good properties that moduli spaces of polarized K3 surfaces enjoy (cf.~e.g.~\cite[Thm.~3.10]{GHS}), though precise statements for moduli stacks do not seem to have appeared in the literature.
\end{remark}

\section{Appendix: motivic integration on algebraic spaces}\label{sec:appendix}
The aim of this section is to prove that one can use models in the category of algebraic spaces to compute motivic integrals.
 Specifically, we will extend the computation of motivic zeta functions on log smooth models from \cite{BuNi} to algebraic spaces; this is required for the proof of Theorem \ref{thm:main}. On the way, we answer a question raised by Stewart and Vologodsky in \cite[A.4(b)]{StVo-arxiv}. In principle, one can go through the entire theory of motivic integration on schemes over discrete valuation rings and check that all the statements remain valid for algebraic spaces. Here, we will use a shortcut instead, passing through the category of formal schemes.

\subsection{Weak N\'eron models and motivic integrals}
\sss Let $R$ be a complete discrete valuation ring with quotient field $K$ and perfect residue field $k$; we do not require $k$ to have characteristic zero.
 Let $X$ be a connected smooth and proper algebraic space over $K$. We define a {\em weak N\'eron model} for $X$ to be a separated smooth algebraic space $\cU$ over $R$, endowed with an isomorphism $\cU_K\to X$, such that for every finite unramified extension $R'$ of $R$ with quotient field $K'$, the map $\cU(R')\to X(K')$ is bijective. For our purposes, we will only need the case where $X$ itself is a scheme; in general, the existence of weak N\'eron models can be proven in exactly the same way as for schemes, by applying the smoothening algorithm in the proof of \cite[3.4.2]{BLR} to a compactification of $\cU$ over $R$ (see \cite{CLO} for an extension of Nagata's embedding theorem to algebraic spaces). The smoothening algorithm is functorial with respect to \'etale morphisms and carries over to algebraic spaces without difficulties.

\sss Let $\omega$ be a volume form on $X$. For every weak N\'eron model $\cU$ of $X$ and every connected component $C$ of $\cU_k$, we can define the order $\ord_C\omega$ of $\omega$ along $C$ in the same way as for schemes: it is the unique integer $m$ such that $\pi^{-m}\omega$ extends to a generator of $\omega_{\cU/R}$ at the generic point of $C$, where $\pi$ is a uniformizer in $R$.

\sss If $Y$ is an algebraic space of finite type over $k$, then $Y$ has a dense open subspace that is a scheme of finite type over $k$ \cite[II.6.8]{knutson}. Thus, by Noetherian induction, we can
partition $Y$ into finitely many subschemes of finite type over $k$. The sum of the classes of these subschemes in the Grothendieck ring of $k$-varieties $K_0(\Var_k)$ does not depend on the chosen partition, so that we can take this sum as the definition of the class $[Y]$ in $K_0(\Var_k)$.
 If $R$ has equal characteristic, we denote by $\mathcal{M}_k$ the localized Grothendieck ring of $k$-varieties $K_0(\Var_k)[\LL^{-1}]$. If $R$ has mixed characteristic, then $\mathcal{M}_k$ will denote the {\em modified} localized Grothendieck ring from \cite[\S3.8]{NiSe-K0}, obtained by trivializing all universal homeomorphisms.
 The key result in this appendix is the following proposition.

\begin{prop}\label{prop:algspint}
Let $X$ be a connected smooth and proper algebraic space over $K$, let $\omega$ be a volume form $X$, and let $\cU$ be a weak N\'eron model for $X$. Then the element
\begin{equation} \label{eq:motint2}
\sum_{C\in \pi_0(\cU_k)}[C]\LL^{-\ord_C\omega}
\end{equation}
 of $\mathcal{M}_k$ only depends on $X$ and $\omega$, and not on the choice of the weak N\'eron model of $X$. In particular, if $X$ is a scheme, then
 $$\int_X|\omega|=\sum_{C\in \pi_0(\cU_k)}[C]\LL^{-\ord_C\omega}.$$
\end{prop}
\begin{proof}
By \cite[4.2.1]{conrad-temkin}, we can consider the analytification $X^{\rig}$ in the category of rigid analytic $K$-varieties. This is a smooth and proper rigid $K$-variety, by \cite[2.3.1]{conrad-temkin}. The volume form $\omega$ on $X$ induces a volume form on $X^{\rig}$ that we will still denote by $\omega$. We claim that the expression \eqref{eq:motint2} is equal to the motivic integral
\begin{equation}\label{eq:motint3}\LL^{\dim(X)}\int_{X^{\rig}}|\omega|
\end{equation}
defined in \cite[4.1.2]{motrigid} -- see also \cite[p.266]{NiSe-motint} for a corrigendum. This implies, in particular, that it does not depend on the choice of $\cU$.

 So let us prove our claim. In order to compute the motivic integral \eqref{eq:motint3}, we construct a formal weak N\'eron model $\mV$ for $X^{\rig}$ in the sense of \cite{formner}. This is a smooth formal $R$-scheme of finite type, endowed with an open immersion of rigid $K$-varieties $\mV_\eta\to X^{\rig}$ that is bijective on $K'$-points for every finite unramified extension of $K$. Then for every connected component $C$ of $\mV_k$, one can define the order $\ord_C\omega$ in exactly the same way as before. By \cite[4.3.1]{motrigid}, we have
 $$\int_{X^{\rig}}|\omega|=\LL^{-\dim(X)}\sum_{C\in \pi_0(\mV_k)}[C]\LL^{-\ord_C\omega}$$ in $\mathcal{M}_k$. The factor $\LL^{-\dim(X)}$ comes from a different choice of normalization of the motivic measure than the one we have made in Section \ref{ss:motint}; the assumption in \cite[4.3.1]{motrigid} that $\mV$ is contained in a formal $R$-model of $X^{\rig}$ is redundant, by \cite[2.43]{NiSe-motint}.

By \cite[II.6.8]{knutson} and the assumption that $k$ is perfect, we can find a partition of $\cU_k$ into finitely many connected $k$-smooth subschemes $U_1,\ldots,U_r$.
 If we denote by $\mU_i$ the formal completion of the algebraic space $\cU$ along $U_i$, then $\mU_i$ is a formal scheme, by \cite[V.2.5]{knutson}.
  The reduction of $\mU_i$ (the closed subscheme defined by the largest ideal of definition $J_{\mU_i}$) is precisely $U_i$; thus it is of finite type over $k$, and $\mU_i$ is a smooth special formal $R$-scheme in Berkovich's terminology used in \cite{Ni,NiSe-motint} (special formal $R$-schemes are also called {\em formally of finite type} in the literature).

  For every $i$ in $\{1,\ldots,r\}$, we denote by $\mV_i\to \mU_i$ the {\em dilatation} centered at $U_i$. This means that $\mV_i$ is the maximal open formal subscheme of the blow-up of $\mU_i$ at $U_i$ such that the ideal $J_{\mU_i}\mathcal{O}_{\mV_i}$ is generated by a uniformizer in $R$. The dilatation satisfies a universal property that guarantees, in particular, that the map $\mV_i(R')\to \mU_i(R')$ is bijective for every finite unramified extension $R'$ of $R$ \cite[2.22]{Ni}. Moreover, $\mV_i$ is a smooth separated
  formal $R$-scheme of finite type, because $\mU_i$ and $U_i$ are smooth (see the proof of \cite[4.15]{Ni}). It follows that the disjoint union $\mV$ of the formal $R$-schemes $\mV_i$ is a weak N\'eron model of $X^{\rig}$. If we denote by $c_i$ the codimension of $U_i$ in $\cU_k$, then it is easy to check that $[(\mV_i)_k]=[U_i]\LL^{c_i}$ in $K_0(\Var_k)$ (see again the proof of \cite[4.15]{Ni}). Moreover, if we write $C_i$ for the unique connected component of $\cU_k$ containing $U_i$, then a straightforward computation also shows that $\ord_{(\mV_i)_k}\omega=\ord_{C_i}\omega+c_i$. Hence,
  $$\LL^{\dim(X)}\int_{X^{\rig}}|\omega|=\sum_{D\in \pi_0(\mV_k)}[D]\LL^{-\ord_D\omega}=\sum_{C\in \pi_0(\cU_k)}[C]\LL^{-\ord_C\omega}$$ in $\mathcal{M}_k$, by the scissor relations in the Grothendieck ring.
\end{proof}

 Thus, when $X$ is a scheme, we can use weak N\'eron models in the category of algebraic spaces to compute motivic integrals of volume forms on $X$. This answers the question raised in \cite[A.4(b)]{StVo-arxiv}.

\subsection{Motivic zeta functions and Nisnevich covers}
\sss From now on, we assume that $k$ is an algebraically closed field of characteristic zero, and we fix an isomorphism $R\cong k[[\pi]]$.
 Let $X$ be a smooth and proper $K$-scheme with trivial canonical line bundle, and let $\omega$ be a volume form on $X$.
 Let $\cX$ be a proper algebraic space over $R$ endowed with an isomorphism of $K$-schemes $\cX_K\to X$. By Proposition \ref{prop:nisnevich}, we can find a partition of $\cX_k$ into subschemes $U_1,\ldots,U_r$
  and, for each $j$ in $\{1,\ldots,r\}$, an \'etale morphism of finite type $\cU_j\to \cX$ such that $\cU_j$ is a scheme and $\cU_j\times_{\cX}U_j\to U_j$ is an isomorphism.
  For each $j$, we consider the motivic zeta function $Z^{\widehat{\mu}}_{\cU_j,\omega}(T)$ as defined in \cite[\S6.2]{BuNi}; here we abuse notation by writing $\omega$ for the restriction of $\omega$ to the generic fiber of $\cU_j$. This zeta function is a formal power series in $T$ with coefficients in the localized Grothendieck ring $\mathcal{M}^{\widehat{\mu}}_{(\cU_j)_k}$ of varieties over $(\cU_j)_k$ with good $\widehat{\mu}$-action. We define the generating series $Z_{j}(T)$ by first applying the base change morphism $$\mathcal{M}^{\widehat{\mu}}_{(\cU_{j})_k}\to \mathcal{M}^{\widehat{\mu}}_{U_j}$$ to the coefficients of $Z^{\widehat{\mu}}_{\cU_j,\omega}(T)$, and then the forgetful morphism $$\mathcal{M}^{\widehat{\mu}}_{U_j}\to \mathcal{M}^{\widehat{\mu}}_{k}.$$

\begin{prop}\label{prop:add}
We have $$Z_{X,\omega}(T)=Z_1(T)+\ldots +Z_r(T)$$ in $\mathcal{M}^{\widehat{\mu}}_{k}$.
\end{prop}
\begin{proof}
The proof is similar to the one of Proposition \ref{prop:algspint}. For every $j$ in $\{1,\ldots,r\}$, we denote by $\mX_j$ the formal completion of $\cX$ along the subscheme $U_j$. It is isomorphic to the completion of $\cU_j$ along $\cU_j\times_{\cX}U_j$. Thus $Z_j(T)$ is precisely the motivic zeta function of the pair $((\mX_j)_\eta,\omega)$, and the result follows from the fact that the rigid varieties  $(\mX_j)_\eta$ form a partition of $X^{\rig}$.
\end{proof}

 \end{document}